\documentclass[11pt,twoside]{amsart}
\usepackage[latin9]{inputenc}
\usepackage{verbatim}
\usepackage{calc}
\usepackage{framed}
\usepackage{amsthm}
\usepackage{amstext}
\usepackage{amssymb}
\usepackage{esint}
\usepackage[unicode=true,pdfusetitle,
 bookmarks=true,bookmarksnumbered=false,bookmarksopen=false,
 breaklinks=false,pdfborder={0 0 1},colorlinks=false]{hyperref}
\makeatletter

\numberwithin{equation}{section}
\numberwithin{figure}{section}
\theoremstyle{plain}
\newtheorem{thm}{\protect\theoremname}
  \theoremstyle{definition}
  \newtheorem{defn}[thm]{\protect\definitionname}
  \theoremstyle{remark}
  
  \theoremstyle{plain}
  \newtheorem{lem}[thm]{\protect\lemmaname}
  \theoremstyle{plain}
  \newtheorem{prop}[thm]{\protect\propositionname}

\usepackage{amssymb,amsthm,amsmath,amsfonts,amscd}
\usepackage{graphicx}
\usepackage{url}
\usepackage{color}
\usepackage[active]{srcltx}
\usepackage[matrix,arrow]{xy}
\usepackage{mathrsfs}
\usepackage{enumerate}
\usepackage{amsopn}
\usepackage{bbm} 
\usepackage{cite}
\allowdisplaybreaks[1]
\usepackage[english]{babel}
\usepackage{bbm}
\usepackage{stmaryrd}
\usepackage{mdef}

\date{\today}

\makeatother
\newcommand {\no}{\noindent}
  \providecommand{\definitionname}{Definition}
  \providecommand{\lemmaname}{Lemma}
  \providecommand{\propositionname}{Proposition}
  \providecommand{\remarkname}{Remark}
\providecommand{\theoremname}{Theorem}

\begin{document}

\title[Long-time behavior, invariant measures  and regularizing effects for SSCL]
{Long-time behavior, invariant measures and regularizing effects for stochastic scalar conservation laws}
\begin{abstract}
\noindent We study the long-time behavior and regularity of the pathwise entropy solutions  to stochastic scalar conservation laws with random in time spatially homogeneous fluxes and periodic initial data. We prove that the solutions converge to their spatial average, which is the unique invariant measure of the associated random dynamical system, and provide a rate of convergence, the latter being  new even in the deterministic case for dimensions higher than two. The main tool is a new regularization result in the spirit of averaging lemmata for scalar conservation laws, which, in particular, implies a regularization by noise-type result for pathwise quasi-solutions. 
\end{abstract}

\author{Benjamin Gess}

\address{Department of Mathematics \\
 University of Chicago \\
 Chicago, IL 60637 \\
 USA }

\email{gess@uchicago.edu}

\author{Panagiotis E. Souganidis }

\address{Department of Mathematics \\
 University of Chicago \\
 Chicago, IL 60637 \\
 USA }

\email{souganidis@math.uchicago.edu }



\thanks{Benjamin Gess is supported by the DFG through a research scholarship. Panagiotis Souganidis is supported by the NSF grant  DMS-1266383.}

\maketitle

\section{Introduction and main results}
\subsection{An overview}
\noindent We are interested in the long-time behavior, the existence and uniqueness of invariant measures and  the regularity of pathwise entropy solutions of the spatially homogeneous stochastic scalar conservation laws (SSCL for short) 
\begin{equation}
\begin{cases}
du+{\displaystyle \sum_{i=1}^{N}\partial_{x_{i}}A^{i}(u)\circ d\b_{t}^{i}=0 \ \text{ in } \ \TT^{N}\times(0,\infty),}\\[1.5mm]
u=u_{0} \ \text{ on } \ \TT^{N} \times \{0\},
\end{cases}\label{eq:scl}
\end{equation}
where $\TT^{N}$ is the $N$-dimensional torus,  $\circ$ denotes the Stratonovich differential, the flux $\bf A$ is smooth with polynomial growth, that is, setting 
$${\bf A}''(\xi):=\partial^2_\xi{\bf A}(\xi) \ \text{ and } \  {\bf a}(\xi):= {\bf A} '(\xi):=\partial _\xi \bf A(\xi),$$
we assume that, for some $C> 0, m \in \N$ and all $\xi \in \R$, 
\begin{equation}\label{a}
{\bf A}=(A^{1},\ldots,A^{N})\in C^{2}(\R;\R^{N}) \ \  \text{and} \ \ 
 |{\bf A}''(\xi)|\le C(1+|\xi|^m),
\end{equation}
\begin{equation}\label{b}
 {\boldsymbol  \b}=(\b^{1},\ldots,\b^{N}) \ \text { is a standard two-sided Brownian motion},
 \end{equation}
 and
\begin{equation}\label{id}
u_0 \in L^1(\TT^N). 
\end{equation}

\noindent  In order to have both nontrivial asymptotic behavior as well as to observe regularizing effects, we need to exclude the linear case by assuming that the flux is genuinely nonlinear. Since our estimates are based on a new stochastic averaging-type lemma (Theorem~\ref{thm:reg} below), it is convenient to quantify this property in a measure theoretic way. 

\noindent We assume that there exist $\t\in(0,1]$ and  $C>0$ such that, 
for all $ { \s} \in S^{N-1}$, ${\ z}\in\R^{N}$ and $\ve>0$,
\begin{equation}\label{flux}
  |\{\xi \in \R: |{\bf a}(\xi){\sigma}-{z}|\le\ve\}|\le C \ve^{\theta}, 
\end{equation}
where $S^{N-1}$ is the unit sphere in $\R^N$ and, for  ${x}=(x_1,\ldots,x_N), {y}=(y_1,\ldots,y_N) \in \R^N,$ we set
 ${ x y}:=(x_1y_1,\dots,x_Ny_N).$
\smallskip

The genuine nonlinearity condition assumed typically in the deterministic setting, that is when $ {\boldsymbol  \b}=(t,\dots,t)$, is that there exist $\t\in(0,1]$ and  $C>0$ such that, 
  for all $ { \s} \in S^{N-1}$, ${z}\in\R$ and $\ve>0$,
  \begin{equation}\label{det-flux}
|\{\xi \in \R: |{\bf a}(\xi)\cdot{ \sigma}-{z}|\le\ve\}|\le C \ve^{\theta},
  \end{equation}
 where $ {x} \cdot { y}$ denotes the inner product of the vectors ${x}, {y} \in \R^N.$ 
 \smallskip 

Since,  for some constant $C> 0$,
 \[|{\bf a}(\xi){\sigma}-{z}|\ge C |{\bf a}(\xi)\cdot{\sigma}-\sum_{i=1}^Nz_i|,\]
 it is immediate that 
 \eqref{det-flux} implies  \eqref{flux}. 
 
The following example shows, however,  that \eqref{flux} is strictly weaker than \eqref{det-flux} in dimensions larger than one. We fix $N=2$ and  some  $l\in\N$ and consider the flux $\mathbf{A}:\R \to \R^2$ given by
   \[
   \mathbf{A}(\xi)=(\frac{1}{l+1}\xi^{l+1}, \frac{1}{l+1}\xi^{l+1}).
   \]
  Then 
  \[
  \mathbf{a}(\xi)=(\xi^{l}, \xi^{l}).
  \]
  In this case, if  $\s=(\frac{1}{\sqrt{2}}, -\frac{1}{\sqrt{2}})$  and $z=0$,  we have  $ \mathbf{a}(\xi)\cdot\sigma-z  =0, $ and, hence, 
   \[
   |\{\xi\in\R:|\mathbf{a}(\xi)\cdot\sigma-z|\le\ve\}|=\infty
   \]
   and \eqref{det-flux} cannot be satisfied. 
   
   \noindent On the other hand, if $\sigma=(\sigma_1,\sigma_2) \in S^1$ and $z=(z_1,z_2)$, 
  \[
   |\mathbf{a}(\xi)\sigma-z|^{2}  =|\xi^l \sigma_1 -z_{1}|^{2}+|\xi^l \sigma_2 -z_{2}|^{2}
   \]
   Since $|\sigma|=1$ yields $\sigma_{1}>\frac{1}{\sqrt{2}}$ or $\sigma_{2}>\frac{1}{\sqrt{2}}$, without loss of generality,  next we assume that  $\sigma_{1}>\frac{1}{\sqrt{2}}$. Then 
   \begin{align*}
   |\{\xi\in\R:|\mathbf{a}(\xi)\sigma-z|\le\ve\}| & =|\{\xi\in\R:|\xi^{l}\sigma_{1}-z_{1}|^{2}+|\xi^{l}\sigma_{2}-z_{2}|^{2}\le\ve^{2}\}|\\
    & \le|\{\xi\in\R:|\xi^{l}\frac{\sigma_{1}}{|\sigma_{1}|}-\frac{z_{1}}{|\sigma_{1}|}|\le\frac{\ve}{|\sigma_{1}|}\}|\\
    & \le\frac{C}{|\sigma_{1}|^{\frac{1}{l}}}\ve^{\frac{1}{l}}\le\ 2^{l/2}C\ve^{\frac{1}{l}},
   \end{align*}
   where the first inequality follows from condition \eqref{det-flux} for the one dimensional flux $a_{1}(\xi)=\xi^{l}$. Hence, \eqref{flux} is satisfied.
\smallskip
 
 \noindent The exponent $\t$ in \eqref{det-flux} depends on the dimension $N$. Indeed, it was shown in  Berthelin  and Junka \cite{BJ10} that, if $\bf A$ is smooth, then necessarily $\t \in (0,\frac{1}{N}]$. 
    
 \noindent In contrast, this is not true for \eqref{flux}. Indeed, as we have seen above, for all $j = 1,\dots,N$,  in general  we have,  
       \begin{align*}
       |\{\xi\in\R:|\mathbf{a}(\xi)\sigma-z|\le\ve\}| & =|\{\xi\in\R:\sum_{i=1}^{N}|a_{i}(\xi)\sigma_{i}-z_{i}|^{2}\le\ve^{2}\}|\\
        & \le|\{\xi\in\R:|a_{j}(\xi)\sigma_{j}-z_{j}|\le\ve\}|.
       \end{align*}
       Hence, if the one dimensional fluxes $a_{j}$ satisfiy \eqref{flux} with exponent $\theta_{j}$, then \textbf{$\mathbf{a}$} satisfies \eqref{flux} with $\theta=\min(\theta_{1},\dots,\theta_{N})$. For example, in the particular case that  $\mathbf{a}(\xi)=(\xi,\dots,\xi)$ which satisfies \eqref{flux} with $\theta=1$, we obtain $\theta=1$ in  \eqref{flux} and there is no dependence on the dimension. 
\smallskip

The main results of the paper are (i)~the quantitative convergence, as $t \to \infty$,  of solutions to \eqref{eq:scl} to the spatial average of the initial data, which turns out to be the unique invariant measure of the associated random dynamical system, (ii)~a new regularizing property (averaging lemma)  for \eqref{eq:scl},  (iii)~a rate for the convergence, as $t \to \infty$, of the deterministic entropy solutions to the their mean, and (iv)~a ``regularization by noise''-type result for pathwise quasi-solutions to \eqref{eq:scl}. 
\smallskip

The convergence to the spatial average, the rate of convergence and the regularizing effect are new results and, to the best of our knowledge, the only available for nonlinear problems with random fluxes. Providing a rate of convergence for deterministic scalar conservation laws with $N>2$ solves a long-standing open problem. Concerning the regularization by noise, we show that pathwise quasi-solutions to the (stochastic) Burgers' equation are more regular than in the deterministic case. Again, to the best of our knowledge, this is the first such result for nonlinear random fluxes.
\smallskip

\subsection{The general setting}
\noindent  Without loss of generality in the following we work with  the filtered probability space $(\O, \mcF,(\mcF_{t})_{t\in\R_+}, \P)$ with the canonical realization on $\O=C_{0}(\R;\R^{N}):=\{{\bf b} \in C(\R; \R^N) \ \text{ and } \ {\bf b}(0)=0 \}$,  and $\P$, $\E$, $\mcF_{t}$ and $\bar{\mcF}_{t}$ being respectively the two-sided standard Gaussian measure on $\O$, the expectation with respect to $\P$, the canonical, uncompleted filtration and its completion. 
\medskip

\noindent  Lions, Perthame and Souganidis  introduced in \cite{LPS13} the notion of  pathwise entropy solutions to \eqref{eq:scl} (actually \cite{LPS13} considered general continuous paths $\boldsymbol \b$) and  showed that, for each $u_{0}\in (L^{1}\cap L^{\infty})(\R^N)$, each path $t\mapsto {\boldsymbol\b}_{t}(\o)$ and all $T>0$, there exists  a unique pathwise entropy solution $u=u(\cdot ; {\boldsymbol \b}, u_{0})=u(\cdot; \omega, u_0)  \in C\big([0,\infty); L^1(\R^N)\big)\cap L^\infty (\R^N \times (0,T)) $  to \eqref{eq:scl} and the solution operator is an $L^1$-contraction and, hence,  is defined on $ L^{1}(\R^N).$\\
\noindent  A straightforward modification of the arguments in \cite{LPS13} yields that the theory extends to \eqref{eq:scl} and is well posed in $L^\infty(\TT^N)$, that is, for each $u_{0}\in L^{\infty}(\TT^N)$, each path $t\mapsto {\boldsymbol\b}_{t}(\o)$ and all $T>0$, there exists  a unique pathwise entropy solution $u=u(\cdot ; {\boldsymbol \b}, u_{0})=u(\cdot; \omega, u_0)  \in C\big([0,\infty); L^1(\TT^N)\big)\cap L^\infty (\TT^N \times [0,T]) $, and the solution operator is an $L^1$-contraction and, hence,  is defined on $ L^{1}(\TT^N).$

\medskip
\noindent Since the entropy solution to \eqref{eq:scl}  is constructed in a pathwise manner,  for each $u_0 \in  L^{1}(\TT^N)$ and $t\geq 0$, the map 
\begin{equation}\label{takis0}
\vp(t,\o)u_0:=u(\cdot, t; \o, u_0) 
\end{equation}
defines a continuous random dynamical system (RDS for short) on $L^{1}(\TT^N)$; we refer to  Appendix~\ref{app:RDS} for some background on RDS. 
\medskip

The associated  Markovian semigroup  $(P_{t})_{t \geq 0}$ 
is given, for each bounded  measurable function $f$ on $L^{1}(\TT^N)$,  $u_0 \in  L^{1}(\TT^N)$ and $t\geq 0$, by 
\[
P_{t}f(u_0):=\E f(u(\cdot,t ;\cdot, u_0))=\E f(\vp(t,\cdot)u_0).
\]
By duality we may consider the action  of $(P_{t})_{t \geq 0}$ on the space $\mcM_1$ of probability measures  on $L^{1}(\TT^N)$, that is, for $\mu\in \mcM_1$,  we set 
$$P_{t}^* \mu (f) := \int_{L^1} P_{t}f(x)d\mu(x).$$

 A probability measure $\mu$ is  an invariant measure for  $(P_{t})_{t \geq 0}$ if 
 $P_{t}^* \mu=\mu$ for all $t\ge 0$.
Moreover  $\mu$ is said to be strongly mixing if, for each $\nu \in \mcM_1$,
  $P_{t}^* \nu \rightharpoonup \mu \  \text{weak}\star \text{ in }\mcM_1$ as $t \to \infty$.
\smallskip

\subsection{The results} We prove here (see Theorem~\ref{main} below) that, as $t\to \infty$ and $\P$-almost surely  (a.s.\ for short),  
$$u(\cdot, t;\omega,u_0) \to \bar{u}_{0}:=\int_{ \TT^{N}}u_{0}(x)dx,$$  
provide a convergence rate and show that  $\d_{\bar{u}_{0}}$  is the unique invariant measure and $\bar{u}_{0}$ the random attractor. Here $\delta_c$ denotes the ``Dirac mass'' measure in $L^1 (\TT^N)$ charging the constant function with value $c\in \R$ and we set $L^1_c(\TT^N)$ to be the space of all $L^1(\TT^N)$ functions with spatial average $c$.
\medskip

\noindent 
The first  result is stated in the next theorem; for some of the terms in the statement  we refer to Appendix~\ref{app:RDS}.
\begin{thm}\label{main}
Assume \eqref{a}, \eqref{b},  \eqref{id}  and  \eqref{flux}. Then, as $t\to \infty$,  
\begin{equation*} 
   u(\cdot, t; \cdot, u_0)\to\bar{u}_0 \ \text{ in $L^{1}(\O;L^{1}(\TT^{N}))$ and $\P$-a.s. in $L^{1}(\TT^{N})$}; 
   \end{equation*}
in particular, $\d_{\bar u_0}$ is the unique invariant measure for $(P_{t})_{t\geq 0}$ on $L^{1}_{\bar u_0}(\TT^{N})$ and is strongly mixing, and, restricted to $L^1_{\bar u_0}(\TT^N)$,  the RDS $\vp$ has ${\bar u_0}$ as forward and pullback random attractor.
Moreover, 
for $t\geq 1$ and $u_0 \in L^{2+m}(\TT^N)$,
\[
\E\|u(\cdot, t;\cdot,u_0)-\bar{u}_0\|_{1}\le \left(\|u_{0}\|_{2+m}^{2+m}+1\right)  t^{-\frac{\t}{3+\t}},
\]
and, for all $p\in (1,\infty)$,
\[
\E\|u(\cdot, t;\cdot,u_0)-\bar{u}_0\|_{p}\le 2\|u_0\|_\infty (\left \|u_{0}\|_{2+m}^{2+m}+1\right) t^{-\frac{\t}{3+\t}}.
\]
\end{thm}
\smallskip

Following the strategy of the proof of Theorem~\ref{main}, we also obtain a rate for the asymptotic behavior of the entropy solutions to the deterministic conservation law
\begin{equation}
\begin{cases}
\partial_t u+{\displaystyle \sum_{i=1}^{N}\partial_{x_{i}}A^{i}(u) =0 \ \text{ in } \ \TT^{N}\times(0,\infty),}\\[1.5mm]
u=u_{0} \ \text{ on } \ \TT^{N}\times\{0\}.
\end{cases}\label{eq:scl-det}
\end{equation}

\no The result is stated next.
\begin{thm}\label{prop:det}
Assume \eqref{a}, \eqref{id}  and  \eqref{det-flux}. Then, as $t\to \infty,$ 
\begin{equation*} 
   u(\cdot, t; \cdot, u_0)\to\bar{u}_0 \ \text{ in } L^{1}(\TT^{N}). 
\end{equation*}
Moreover, 
for $t\geq 1$ and $u_0 \in L^{2+m}(\TT^N)$,
\[
\|u(\cdot, t;\cdot,u_0)-\bar{u}_0\|_{1}\le \left(\|u_{0}\|_{2+m}^{2+m}+1\right) t^{-\frac{\t}{2+\t}}.
\]
\end{thm}

\no The rate in Theorem~\ref{prop:det} is a new result for $N>2$. We give some details on the existing literature for the \eqref{eq:scl-det} in the next subsection.
\smallskip

We do not make any claim of optimality for the rates obtained in Theorem \ref{main} and Theorem \ref{prop:det}. Moreover, although the rate for the stochastic problem  is better than the one for the deterministic one, it remains an open question how the optimal rates compare.
\smallskip

\no The arguments and estimates leading to Theorem \ref{main} are a special case of a more general new regularity result, which  extends the ones  obtained by Lions, Perthame and Souganidis in \cite{LPS12}.
\begin{thm}\label{thm:reg}
  Assume \eqref{a}, \eqref{b},  \eqref{id}, \eqref{flux} and let $u_0 \in L^{2+m}(\TT^N)$. Then, for all $\l\in (0, \frac{4\t}{2\t+3})$ and $T>0$, there is a $C>0$ such that
$$    \E\int_0^T \|u(t)\|_{W^{\l,1}} dt
    \le C (1+ \|u_{0}\|_{2+m}^{2+m} ), $$
and, for all  $\d>0,$
    $$ \sup _{t \ge \d} \E\|u(t)\|_{W^{\l,1}} < \infty. $$ 
\end{thm}
We note that Theorem \ref{thm:reg} yields higher  regularity than in the corresponding deterministic result. Indeed, Jabin and Perthame proved in \cite{JP02} that, if the genuine nonlinearity condition \eqref{det-flux} holds for  $\t=1$, then, for $t>0$, entropy solutions to \eqref{eq:scl-det} satisfy $u(t)\in W^{\l,1}$ for each $\l\in(0,\frac{1}{3})$. In contrast, Theorem \ref{thm:reg} yields that, for $t>0$,  $u(t)\in W^{\l,1}$ for each $\l\in(0,\frac{4}{5})$. As before, we do not know if the 
the regularity obtained in Theorem \ref{thm:reg} is optimal.
\smallskip

Theorem \ref{thm:reg} also  implies a regularization by noise-type result in the following sense. In analogy to the deterministic theory, we introduce in  Section \ref{sec:regularization} the class of quasi-solutions to the stochastic Burgers equation \begin{equation}\label{eq:stoch_burgers}
\begin{cases}
du+ \partial_{x}u^{2}\circ d\b_{t}= 0 \ \text{ in } \ \TT\times(0,T),\\[1mm]
u=  u_{0} \ \text{ on } \ \TT\times\{0\},
\end{cases}
\end{equation}
and show that they are more regular than the ones to 
\begin{equation}\label{eq:det_scl}
\begin{cases}
\partial_tu+ \partial_{x}u^{2}=  0 \ \text{ in } \ \TT\times(0,T),\\[1mm]
u=  u_{0} \ \text{ on } \ \TT\times\{0\}.
\end{cases}
\end{equation}

The reason we work with such solutions is that the regularity implied by averaging techniques for \eqref{eq:det_scl} is essentially sharp for quasi-solutions, a fact which is not true for entropy solutions; see, for example, De Lellis and Westdickenberg \cite{DLW03} and De Lellis, Otto and Westdickenberg \cite{DLOW03}. Indeed following the proof in \cite{JP02} it easy to conclude that, for $t>0$,  quasi-solutions to \eqref{eq:det_scl} satisfy $u(t)\in W^{\l,1}$ for every $\l\in(0,\frac{1}{3})$; actually \cite{JP02} works with entropy solutions, but a careful look at the proof yields that the result also holds for quasi-solutions. Optimal regularity for quasi-solutions was proven by Golse and Perthame in \cite{GP13}. As it shown in \cite{DLW03}, however, there are quasi-solutions such that $u(t)\not\in W^{\l,1}$ for every $\l>\frac{1}{3}$.
\smallskip

In contrast, we show here that quasi-solutions to  \eqref{eq:stoch_burgers} satisfy $u(t)\in W^{\l,1}$ for every $\l\in(0,\frac{4}{5})$. In this sense, we see that the noise included in \eqref{eq:stoch_burgers} has a regularizing effect. 

\begin{thm}\label{thm:regularization}
Let u be a pathwise quasi-solution to \eqref{eq:stoch_burgers} with $u_{0}\in L^{2}(\TT)$. Then, for all $\l\in(0,\frac{4}{5})$ and $T>0$, there is $C>0$ such that
\[
\E\int_{0}^{T}\|u(t)\|_{W^{\l,1}}dt\le C(1+\|u_{0}\|_{2}^{2}+\E|m|([0,T]\times\TT\times\R)).
\]
and, for all $\d >0,$ 
$$\sup_{t \ge \d}\E\|u(t)\|_{W^{\l,1}}<\infty.$$ 
\end{thm}

We expect, and this is the  subject of a subsequent work, that methods similar to the ones developed in this paper, in combination with arguments from \cite{DV13}, may be used to prove the existence and uniqueness of an invariant measure in the case of additive noise, that is for
$$
du+ \sum_{i=1}^{N}\partial_{x_{i}}A^{i}(u)\circ d\b_{t}^{i}= dW_t,
$$
where $W$ is an infinite dimensional Wiener process independent of $\b$. 

\subsection{Brief review of the existing literature}
\noindent We discuss next briefly results about  scalar conservation laws with some type of random dependence and their long-time behavior, we recall the basic literature for  averaging lemmata and touch upon the issue of regularization by noise.
\medskip

\noindent The effect of noise entering scalar conservation laws via randomness in the initial condition has been studied, for example, by Avellaneda and E \cite{AE95}, Burgers \cite{B74}, Ryan\cite{R98} and Sinai\cite{S92-2}.
\smallskip

For stochastic scalar conservation laws driven by additive noise, also including boundary value problems, we refer to the works by E, Khanin, Mazel and Sinai \cite{EKMS00}, Kim \cite{K03}, Nakazawa \cite{N82}, Saussereau and Stoica \cite{SS12}, Vallet and  Wittbold \cite{VW09} and the references therein. 
\smallskip

The case of multiplicative semilinear noise dependence in It\^o's form, that is stochastic PDE (SPDE for short) of the form
\begin{equation}
du+\div A(u)dt=g(u)dW_{t},\label{eq:semilinear_SSCL}
\end{equation}
where $W$ denotes a possibly infinite dimensional Wiener process, has attracted considerable interest in recent years;  see for example, Bauzet, Vallet and Wittbold \cite{BVW13}, Chen, Ding and Karlsen,  \cite{CDK12}. Debussche and Vovelle \cite{DV10},  Debussche,  Hofmanova and Vovelle \cite{DHV13},  Feng and Nualart\cite{FN08}, Hofmanova \cite{H13}, Holden and Risebro\cite{HR97}, etc.. 
\smallskip

The long-time behavior of solutions to \eqref{eq:semilinear_SSCL}, the  existence and uniqueness of an invariant measure and  a singleton random point attractor was first studied in \cite{EKMS00} for the stochastic Burgers equation with additive noise, that is 
\begin{equation}
du+\partial_{x}u^{2}dt=dW_{t},\label{eq:semilinear_SSCL-1}
\end{equation}
on the one-dimensional torus $\TT=[0,2\pi]$;  for extensions of these results see Bakhtin\cite{B13-2}, Boritchev \cite{B13-3}, Iturriaga and Khanin \cite{IK03}, Saussereau and Stoica \cite{SS12}, etc..
\smallskip

More recently, these results were  extended by Debussche and Vovelle \cite{DV13} to general SSCL on $\TT^{N}$ with additive noise, that is to 
\begin{equation}
du+\div A(u)dt=\Phi dW_{t},\label{eq:semilinear_SSCL-additive}
\end{equation}
where $\Phi$ is an appropriate Hilbert-Schmidt operator. In addition, it was  shown in \cite{DV13} that, under appropriate non-degeneracy conditions on the flux $\bf A$, \eqref{eq:semilinear_SSCL-additive} has a unique invariant measure. 
\medskip

We also refer to Dirr and Souganidis \cite{DS05}, who studied similar questions for Hamilton-Jacobi equations perturbed by additive noise in any dimension.
\medskip

\noindent All of the above mentioned works consider semilinear SSCL in the sense that the noise is applied to functions of the solution and not its derivatives. 
\smallskip

In contrast, \cite{LPS13} and its subsequent extensions by Lions, Perthame and Souganidis \cite{LPS14} and the authors \cite{GS14} consider SSCL like
\[
du+\sum_{i=1}^{N}\partial_{x_i}A_{i}(x,u)\circ d\b_{t}^{i}=0,
\]
with inhomogeneous spatially  dependent fluxes and single and multiple rough time dependence respectively, in the sense that the noise is applied to the flux and, hence,  the derivatives of the solution.

\medskip
The  well-posedness of solutions to such SPDE was proven in \cite{LPS12,LPS13,GS14} using a kinetic formulation. The case of joint transport noise and linear multiplicative noise has been treated by Friz and  Gess in \cite{FG14} by entropy and rough paths methods.
\smallskip

We next discuss some of the available literature about the  long-time behavior of deterministic scalar conservation laws (SCL) with periodic initial data. Since there are many references, here we restrict to the ones  that seem most relevant for our purpose. The one-dimensional setting is well understood, since Lax in \cite{L57} proved, using the Lax-Oleinik formula, asymptotic linear decay for strictly convex or concave fluxes. The first asymptotic result with a rate for  $N=2$ is due to Engquist and  E in \cite{EE93}. The higher dimensional case  was first studied  by Chen and  Frid in \cite{CF99} and was subsequently generalized by Chen and Perthame \cite{CP09} and Debussche and Vovelle \cite{DV09}. These last three references assumed genuine non-linearity-type conditions on the flux and proved the  convergence of the solutions to their spatial average without any rate. More recently and  independently, Dafermos \cite{D13} and Panov \cite{P13} proved the same result under weaker genuine non-linearity conditions, which are also necessary. We note that with the exception of $N=1,2$ no rate was known before.
\smallskip

Many of the above mentioned results about the long-time behavior of deterministic SCL rely on averaging lemmata of the type introduced by Golse, Lions, Perthame, Sentis \cite{GLPS88}. Averaging lemmata are one of the most important tools in the theory of (deterministic) conservation laws and were used very effectively for conservation laws by, among others, Perthame and Tadmor \cite{PT91}, Lions, Perthame and Tadmor \cite{LPT94}, Lions, Perthame and Souganidis \cite{LPS96}, Perthame and Souganidis \cite{PS98} etc.. Some stochastic averaging lemmata, which lead to higher regularizing results than in the corresponding deterministic case, were obtained by  Lions, Perthame and Souganidis \cite{LPS12}.  
\smallskip

Regularization by noise phenomena  have been observed in the linear spatially inhomogeneous case, that is for transport equations of the form
\begin{equation}
du+b(x)\cdot Du\ dt=0.\label{eq:det_transp}
\end{equation}
Indeed, it has been shown by Flandoli, Gubinelli and Priola \cite{FGP10} that adding perturbations of the type
\begin{equation}
du+b(x)\cdot Du\ dt=\partial_{x}u\circ d\b_{t},\label{eq:stoch_trans}
\end{equation}
regularizes the dynamics in the sense that well-posedness can be obtained for \eqref{eq:stoch_trans} in cases for which \eqref{eq:det_transp} is ill-posed. Moreover, it has been proven (see, for example, Fedrizzi and Flandoli \cite{FF13}, Flandoli, Gubinelli and Priola \cite{FGP13} and Mohammed, Nilssen and Proske \cite{MNP13}), that solutions to \eqref{eq:stoch_trans} are more regular  than their deterministic analogues. For example, \cite{FF13} shows that, if  $u_{0}\in\bigcap_{\l\ge1}W^{\l,1}$ and $b\in L^{p}(\R)$ with $p\ge2$,  then $u(t)\in\bigcap_{\l\ge1}W_{loc}^{\l,1}$ a.s. and for every $t\in[0,T]$. 
\smallskip

All the above results and techniques depend strongly on the linear structure of \eqref{eq:stoch_trans}. In fact, Flandoli  \cite{F11} showed that analogous regularizing effects are wrong in the nonlinear case, for example, the stochastic Burgers' equation
\begin{align*}
du+
\partial_{x}u^{2}dt= &  \partial_{x}u\circ d\b_{t},
\end{align*}
where the inclusion of noise does not prevent characteristics to collide and thus BV-regularity is the best one can get even  in the stochastic case.

\subsection{Organization of the paper and some notation}
The paper is organized as follows: In Section~\ref{sec2} we recall the notion of stochastic pathwise entropy as well as some of the basic estimates that we need to prove the main results. Theorem~\ref{main}, Theorem~\ref{prop:det} and Theorem~\ref{thm:reg} are proved respectively  in Section~\ref{sec3}, Section~\ref{sec:det_case} and Section~\ref{sec4}. In  Section~\ref{sec:regularization}  we recall the notion of quasi-solutions  to \eqref{eq:det_scl}, introduce the definition of pathwise quasi-solutions to \eqref{eq:stoch_burgers} and prove Theorem~\ref{thm:regularization}.  Appendix~\ref{app:RDS} states some of the basic properties of RDS, in  Appendix~\ref{lemma} we present the proof of a very useful and basic analysis lemma which  in the proofs, and finally, in Appendix~\ref{app:splitup} we give the rigorous justification of a formula which is at the core of the proofs. 
\smallskip

In what follows,  we will say pathwise instead of stochastic pathwise entropy solution, we will omit, when it does not cause any confusion, the dependence in $\omega$ and we occasionally write 
$m(x,\xi,s)dx d\xi ds$ instead of $dm(x,\xi,s).$  For notational convenience, we write $A \lesssim B$  if $A \leq CB \ \hbox{ for some } C>0.$  If $A \lesssim B$ and $B \lesssim A$, we write $A \sim B$. We work with the homogeneous Bessel potential spaces $W^{\l,p}$ for $\l > 0$, $p\in [1,\infty)$, that is 
  $$W^{\l,p}: = \{f\in L^p(\TT^N):\ (|n|^\l \hat f(n))^\vee \in L^p(\TT^N)\},$$
where $\hat f$  denotes  the discrete Fourier transform of $f$ on $\TT^N$ and $f^\vee$ its inverse. We note that the homogeneous Bessel potential spaces coincide with the domains of the fractional Laplace operators $(-\D)^\frac{\l}{2}$ on $L^p(\TT^N)$. For notational simplicity we also set $H^\l := W^{\l,2}$.

\section{Pathwise entropy solutions and some basic properties}\label{sec2}

\subsection{The definition}\label{sec:defn_pathwise_entropy}
We begin with the derivation of the notion of the pathwise solution to \eqref{eq:scl} and, for the moment,  we assume that $\boldsymbol \b$ is smooth in which case $\circ$ reduces to multiplication. Then \eqref{eq:scl} is a standard scalar conservation law with time dependent flux which is best studied using kinetic solutions. 
\smallskip 

For the latter we introduce the auxiliary nonlinear  function
\begin{equation}
\chi(x,\xi,t)=\chi(u(x,t),\xi)=\left\{ \begin{array}{l}
+1 \ \text{ if } \ 0\leq\xi\leq u(x,t),\\[1.5mm]
-1 \ \text{ if } \ u(x,t)\leq\xi\leq0,\\[1.5mm]
\;0 \ \text{ otherwise}.
\end{array}\right.\label{eq:char_fctn}
\end{equation}
The kinetic formulation  of \eqref{eq:scl} is the assertion that the $\chi$ given by \eqref{eq:char_fctn} is a solution, in the sense of distributions, to 
\begin{equation}\label{kinetic form-1}
\begin{cases}
\partial_{t}\chi+\sum_{i=1}^{N} a^{i}(\xi)\partial_{x_{i}}\chi\dot{\b}^{i}(t) =\partial_\xi m \ \text{ in } \ \TT^{N}\times\R\times(0,T),\\[1.5mm]
\chi  =\chi_0(\cdot,\cdot):=\chi(u_{0}(\cdot),\cdot) \ \text{ on }\TT^{N}\times\R\times\{0\},
\end{cases}
\end{equation}
where 
$$m \ \text{ is a nonnegative bounded measure on $\TT^{N} \times \R\times (0,T)$ for each $T>0$.}$$

\smallskip

To introduce the pathwise entropy solutions we use the transport equation
\begin{equation}\label{eq:transport_1}
\begin{cases}
\partial_t\vr+\sum_{i=1}^{N}a^{i}(\xi)\partial_{x_{i}}\vr\dot{\b}^{i}(t)  =0 \  \text{ in \ $\TT^N\times \R\times (0,\infty)$}, \\[1mm]
\vr  =\vr_{0}  \ \text{ on }\TT^{N}\times\R\times\{0\}.
\end{cases}
\end{equation}

For each $y \in \TT^N$ and $\rho_0 \in C^\infty(\TT^N)$, the solution   $\vr=\vr(x,y,\xi,t)$  to \eqref{eq:transport_1} with initial condition $\vr_{0}(\cdot-y)$ is given by 
\begin{equation}\label{takis10000}
\vr(x,y,\xi,t)=\vr_0\left(x-y-\sum_{i=1}^{N}a^{i}(\xi){\b}^{i}(t)\right).
\end{equation}
We define the ``convolution along characteristics'' $\vr\ast\chi:\TT^N \times \R \times [0,\infty) \to \R$ by
\[
\vr\ast\chi(y,\xi,t):=\int_{\TT^N}\vr(x,y,\xi, t)\chi(x,\xi,t)dx.
\]
It follows that,   in the sense of distributions in $\xi$ and $t$,
\begin{equation}\label{takis1} 
\partial_{t}(\vr\ast\chi)(y,\xi,t)=\int_{\TT^N} \vr(x,y,\xi,t)\partial_{\xi} m(x,\xi, t)dx\ \ \text{in }\TT^{N}\times\R\times(0,T),
\end{equation}
and, after integrating in $t$, again in the the sense of distributions in $\xi$ and for all $t\in [0,T]$,
\begin{align*}
(\vr\ast\chi)(y,\xi,t) - (\vr\ast\xi)(y,\xi,0) =  \int_{0}^{t}\int_{\TT^N}\vr(x,y,\xi,r)\partial_{\xi}m(x,\xi,r)dxdr.
\end{align*}
Note that this last identity is actually equivalent to the kinetic formulation for conservation laws with  smooth time dependence. Moreover, the dependence on the derivatives of the paths has disappeared. These two observations  are the  idea behind the notion of pathwise entropy solutions, that is, to use  \eqref{takis1} as the definition of a solution.  To state the latter we need to introduce one more notion.
\smallskip

We say that a measurable map $m: \Omega \to \mcM$, the space of nonnegative bounded measures on $\TT^N\times\R\times[0,T]$, is a kinetic measure, if the process $t\mapsto m(\cdot, [0,t])$  with values in the space of nonnegative bounded measures on $ \TT^N \times \R $ is $\mcF_t$-adapted.

\begin{defn}
\label{def:path_e-soln} A map $u:\TT^N \times [0,\infty) \times \Omega \to \R$  such that, for all $T>0$,  $u\in  (L^{1}\cap L^\infty)(\TT^N \times [0,T]\times \Omega)$, and $\P$-a.s. in $\omega$, $u(\cdot, \omega)  \in C \big([0,T]; L^1(\TT^N)\big)\cap L^\infty (\TT^N \times [0,T])$, and $t\mapsto u(t, \cdot )$ is an $\mcF_{t}$-adapted  process in $L^1(\TT^N)$, is a (stochastic) pathwise entropy solution to \eqref{eq:scl}, if there exists a kinetic measure $m$ such that, for all $y\in\TT^N$ and all test functions $\vr$ given by \eqref{takis10000} with $\vr_{0}\in C^{\infty}(\TT^N)$ and all  $\vp\in C_{c}^{\infty}([0,T)),\psi\in C_c^\infty(\R)$, 
\begin{align*}
 & \int_{0}^{T}\int_{\R}\partial_{t}\vp(r)\psi(\xi)(\vr\ast\chi)(y,\xi,r)drd\xi+\int_{\R}\vp(0)\psi(\xi)(\vr\ast\chi)(y,\xi,0)d\xi\\
 & =\int_{0}^{T}\int_{{\TT^N}\times \R}\vp(r)\partial_{\xi}\left(\psi(\xi)\vr(x,y,\xi,r)\right)dm(x,\xi,r).
\end{align*}
\end{defn}

\subsection{The basic properties}
The next proposition summarizes the key properties of the pathwise entropy solution and is obtained by a straightforward modification of the proof of the analogous result 
in   \cite{LPS13}.
\begin{prop}
\label{prop:well-posedness}  Assume \eqref{a}, \eqref{b} and   \eqref{id}.  For each $u_{0}\in L^{\infty}(\TT^N)$ and each path $t\mapsto {\boldsymbol\b}_{t}(\o)$, there exists a unique stochastic pathwise entropy solution $u=u(\cdot, \omega; u_{0})$. 
Moreover, for all $t\geq 0$, $p\in[1,\infty]$ and  $\P$-a.s., 
\begin{align}
\int u(x,t,\omega;u_{0})dx&=\int u_{0}(x)dx, \label{eq:mean-preservation} \\
\|u(\cdot, t, \omega; u_{0})\|_{p}&\le\|u_{0}\|_{p},\label{eq:infty-bound} 
\end{align}
and
\begin{equation}
\|u(\cdot, t, \omega; u_{0})\|_{BV}\le\|u_{0}\|_{BV}.\label{eq:BV-bound},
\end{equation}
and, if $u^{1}(\cdot, \cdot;\omega), u^{2}(\cdot, \cdot;\omega)$ are two pathwise  entropy solutions, then, for all $t, s \geq 0$ with $s\leq t$ and  $\P$-a.s., 
\begin{equation}
\|u^{1}(\cdot, t, \omega)-u^{2}(\cdot,t,\omega)\|_{L^{1}(\TT^N)}\le\|u^{1}(\cdot,s, \omega)-u^{2}(\cdot,s,\omega) \|_{L^{1}(\TT^N)}.\label{eq:pathwise_contraction}
\end{equation}
\end{prop}
\smallskip

Note that  Proposition \ref{prop:well-posedness} and, in particular, \eqref{eq:pathwise_contraction}, yield that the solutions $u=u(\cdot, \omega; u_{0})$ are well-defined for all $u_0\in L^1(\TT^N)$. It  is then easy to see that the map  defined in \eqref{takis0} is indeed a continuous RDS on $L^{1}(\TT^{N})$. 
\smallskip

For future  use we remark that, in view the definition of $\chi$, for each $u\in \R$,
 it holds $\int_\R |\chi(u,\xi)|d\xi =  |u|$ and, thus, for any $u\in L^1(\TT^N)$ and $p\ge1$, we have
\begin{align*}
  \|\chi(u(\cdot),\cdot)\|_{L^p(\TT^N\times\R)}^p &\le \|\chi(u(\cdot),\cdot)\|_{L^1(\TT^N\times\R)} = \int_{{\TT^N}\times \R} |\chi(u(x),\xi)|d\xi dx\\&=\int_{\TT^N} |u(x)|dx=\|u\|_{L^1(\TT^N)}.
\end{align*}

In the sequel we will also need the following bound on the mass of the entropy defect measure $m$, which a small extension of the analogous bounds in  \cite{LPS13, LPS14, GS14}, where we refer to for the proofs. 
\begin{lem}
\label{lem:m-bound}  Assume \eqref{a}, \eqref{b} and   \eqref{id}.  For $u_{0}\in L^\infty(\TT^N)$, $m\in \N$, $t\geq 0$ and $\P$-a.s., 
\begin{align*}
\|u(\cdot, t,\omega, u_0)\|_{m+2}^{m+2} & + (m+2)(m+1)\int_{0}^{t}\int_{{\TT^N}\times \R} |\xi|^{m} dm(x,\xi,s)=\|u_0\|_{m+2}^{m+2},
\end{align*}
and  
\begin{align*}
 & (m+2)(m+1)\E\int_0^\infty \int_{{\TT^N}\times \R} |\xi|^{m} dm(x,\xi,s)\le\|u_0\|_{m+2}^{m+2}.
\end{align*}
\end{lem}

\section{The quantitative asymptotic behavior.}\label{sec3}

The key step here is to introduce a regularization in the spirit of the classical averaging lemmata, which allows to obtain estimates that control the long time behavior of the solution to \eqref{eq:scl}. The typical  proofs of such averaging lemmata employ space time Fourier transforms, which is not possible in our context  because of the rough time dependence of the flux. Instead, here we 
use only Fourier transforms with respect to the space variable $x$, a technique developed in Bouchut and Desvillettes \cite{BD99} and used in Debussche and Vovelle \cite{DV13} as well as in  \cite{LPS12}. Although we follow the arguments of \cite{DV13}, here we need to deal with the new difficulties arising because of the stochastic nature of the flux in \eqref{eq:scl}. An important tool  is a technical result (Lemma~\ref{lem:integral_estimate}), which uses the genuine nonlinearity condition \eqref{flux}. Since the proof is rather long, we divide it in several subsections.

\subsection{Averaging lemmata, a regularization result and a ``new formula'' for the solution} 
In view of  \eqref{eq:pathwise_contraction} and the density of $L^\infty(\TT^N)$ in $L^1(\TT^N)$, it is clear that it suffices to consider $u_0 \in L^\infty(\TT^N)$.  

\smallskip
Moreover, 
to keep the presentation simple, throughout this section we restrict the presentation to the zero average case, that is, henceforth, we assume that  
\begin{equation}\label{takis2}
\int u_{0}(x)dx=0
\end{equation}
and leave it up to the reader to fill in the details for the general case, which is easily reduced to \eqref{takis2}.
\vskip.075in

The main idea is that averaging over the ``kinetic variable'' $\xi$ of the solution $\chi$  to \eqref{eq:transport_1} yields more regularity and provides new estimates
for the solutions to the conservation laws. The difficulty, of course, comes from the very low regularity of the right hand side $\partial_\xi m$ in \eqref{kinetic form-1}. Typically this problem is dealt with by adding a regularizing (parabolic)-type term in the kinetic equation and taking advantage of the smoothing it generates.
\smallskip

In this paper we use the fractional Laplacian operator as a regularizing term (see \cite{BD99} and \cite{DV13} for similar types of arguments)
 \begin{equation}\label{takis3}
 B:=(-\D)^{\alpha}+Id \ \text{ with} \ \a\in(0,1],
 \end{equation}
and, for $\gamma >0$, we   rewrite \eqref{kinetic form-1} as 
\begin{equation}\label{eq:kinetic_form-1-1}
\begin{cases}
\partial_t\chi+\sum_{i}a^{i}(\xi)\partial_{x_{i}}\chi \circ d{\b}^{i}(t)+\gamma B\chi =\gamma B\chi+\partial_\xi m  \ \text{ in } \  \TT^{N}\times\R\times(0,T),
\\[1.4mm]
\chi  =\chi(u_{0}(\cdot),\cdot)\text{ on }\TT^{N}\times\R\times\{0\}.
\end{cases}
\end{equation}

Note that in what follows we make strong use of the properties of the Gaussian paths and our approach and results do not extend to  general continuous driving signals replacing $\boldsymbol{\b}$.
\smallskip

Let $S_{A_{\gamma}(\xi)}(s,t)$ denote the solution operator (group) of 
\begin{equation}\label{eq:semigroup-eqn}
\partial_t v + \sum_{i}a^{i}(\xi)\dot{\b}^{i}(t)\partial_{x_{i}}v +\gamma Bv = 0 \ \text{in} \ \TT^N\times \R \times (s,\infty).
\end{equation}
It is immediate that, for all $f$ in the appropriate function space,  
\begin{equation}\label{eq:semigroup}
  S_{A_{\gamma}(\xi)}(s,t)f(x)= (S_{\g B}(t-s)f)\left(x-\mathbf{a}(\xi)(\boldsymbol \b(t)-\boldsymbol \b(s))\right),
\end{equation}
where $S_{\g B}(t)$ is the solution semigroup to 
\begin{equation}\label{takis101}
\partial_t v + \gamma B v=0 \  \text{ in} \ \R^N\times (0,\infty),
\end{equation}
and 
$$ \mathbf{a}(\xi)(\boldsymbol \b(t)-\boldsymbol \b(s))= \big(a^1(\xi)(\beta^1(t)-\beta^1(s)), \dots, a^N(\xi)(\beta^N(t)-\beta^N(s))\big).$$
Note that, for each $n\in\Z^N$, the Fourier transform  of  $S_{A_{\gamma}(\xi)}$ corresponds to multiplication by $$\exp({-i {\bf{a}(\xi)} (\boldsymbol{\beta}(t)-\boldsymbol {\beta}(s)) \cdot n
-\gamma (|n|^{2\alpha} + 1)(t-s))}.$$

It follows (see Appendix \ref{app:splitup} for a rigorous justification)  from the variation of parameters formula that the solution $\chi$ to \eqref{takis3} is given, in the sense of distributions in $x$ and $\xi$,  by
\begin{equation}\label{vc100}
\chi(x,\xi,t)=  S_{A_{\gamma}(\xi)}(0,t)\chi_{0}(x,\xi)
  +\int_{0}^{t}S_{A_{\gamma}(\xi)}(s,t)(\gamma B\chi(s)+m_{\xi}(x,\xi, s))ds.
\end{equation}

After integrating in  $\xi$, we find that, for  $t \in [0,T]$ and  $\vp \in C^\infty(\TT^N)$,
\begin{equation}\begin{split}\label{eq:splitup}
\int_{\TT^N} \vp(x)u(x,t)dx= & \int_{{\TT^N} \times \R}\vp(x)\chi(x,\xi,t)dxd\xi\\
= & \int_{{\TT^N} \times \R}\vp(x)S_{A_{\gamma}(\xi)}(0,t)\chi_{0}(x,\xi)dxd\xi\\
 & +\int_{0}^{t} \int_{{\TT^N} \times \R}\gamma B(S^*_{A_{\gamma}(\xi)}(s,t)\vp)(x)\chi(x,\xi,s)dxd\xi ds\\
 & -\int_{0}^{t}\int_{{\TT^N}\times \R} \partial_{\xi}(S^*_{A_{\gamma}(\xi)}(s,t)\vp)(x)dm(x,\xi,s),
\end{split}\end{equation}
where $S^*_{A_{\gamma}(\xi)}$ denotes the adjoint semigroup to $S_{A_{\gamma}(\xi)}$.
\smallskip

Accordingly, in the sense of distributions in $x$,  we write (recall that, a.s. in $\omega$, $u \in C([0,\infty); L^1(\TT^N))$)
\begin{equation}\label{takis5}
u(t)=  u^{0}(t)+u^{1}(t)+Q(t),
\end{equation}
where, for $\vp \in C^\infty(\TT^N)$, 
\begin{align}
\<u^0(t),\vp\>&
:=\int_{{\TT^N}\times \R}  \vp(x) S_{A_{\gamma}(\xi)}(0,t)\chi_{0}(x,\xi)dxd\xi, \label{takis6}\\
\<u^1(t),\vp\>&
:=\int_{0}^{t}\int_{{\TT^N}\times \R} \gamma(B(S^*_{A_{\gamma}(\xi)}(s,t)\vp)(x)\chi(x,\xi,s)dxd\xi ds \label{takis7}
\end{align}
and
\begin{equation}
\<Q(t),\vp\> :=  -\int_{0}^{t}\int_{{\TT^N}\times\R}\partial_{\xi}(S^*_{A_{\gamma}(\xi)}(s,t)\vp)(x)m(x,\xi,s)dxd\xi ds\label{takis8}.
\end{equation}

Next, we estimate each of \eqref{takis6}, \eqref{takis7} and \eqref{takis8} separately using averaging techniques. 
\smallskip

In the analysis we need a basic integral estimate which is proved in Appendix~\ref{lemma}. For its statement  it is convenient to introduce, for each  $b:\R\to\R^{N}$ and $f\in L^{2}$,
the function $\phi(\cdot; b,f):\R^N \to \R$ given by
\begin{equation}\label{takis9}
\phi(w; b,f):=e^{-\frac{|w|^{2}}{a}}\int_\R e^{ib(\xi)\cdot w}f(\xi)d\xi.
\end{equation}

\begin{lem}
\label{lem:integral_estimate}Let $b:\R\to\R^{N}$ be such that, for all $\ve>0,z\in\R^{N}$ and some nondecreasing $\iota: [0,\infty) \to [0,\infty)$ 
with  $\lim_{\ve \to 0}i(\ve)=0$,
\[
|\{\xi \in \R :|b(\xi)-z|\le\ve\}|\le\iota(\ve).
\]
Then, for all $a>0$ and $f\in L^{2}(\R)$, 
\[ \|\phi(\cdot;b,f)\|^2_{L^{2}} \le2\sqrt{a\pi}\int_{0}^\infty \tau e^{-\tau^{2}}\iota(\frac{2\tau }{\sqrt{a}})d\tau\|f\|_{2}^{2}.
\]
\end{lem}

\subsection{The estimate of  \texorpdfstring{$u^{0}$}{u\^\{0\}}} Taking Fourier transforms  in \eqref{takis6} yields, for each $n\in\Z^N$, 
\[
\hat{u}^{0}(n,t)=\int e^{-i\mathbf{a}(\xi)\boldsymbol{\b}(t)\cdot n-\gamma(|n|^{2\alpha}+1)t}\hat{\chi}_{0}(n,\xi)d\xi.
\]

It is immediate, in view of \eqref{takis2},  that 
\[
\hat{u}^{0}(0,t)= \int e^{-\g t}\hat \chi_{0}(0,\xi)d\xi= \int \int e^{-\g t}  \chi_{0}(x, \xi)d\xi dx=\int e^{-\g t} u^{0}(x)dx=0.
\]

When  $n\in \Z^N\setminus \{0\}$,  integrating in time, taking expectations and using the scaling properties of the Brownian paths, we find
\begin{align*}
\E\int_{0}^{T}|\hat{u}^{0}(n,t)|^{2}dt 
& =\E\int_{0}^{T}\Big|\int e^{-i\mathbf{a}(\xi)\boldsymbol {\b}({t})\cdot n-\gamma(|n|^{2\alpha}+1)t}\hat{\chi}_{0}(n,\xi)d\xi\Big|^{2}dt\\
 & =\int_{0}^{T}e^{-2\g(|n|^{2\alpha}+1)t}\E\Big|\int e^{-i\mathbf{a}(\xi)\boldsymbol{\b}({t|n|^{2}})\cdot\frac{n}{|n|}}\hat{\chi}_{0}(n,\xi)d\xi\Big|^{2}dt\\
 & =\int_{0}^{T}e^{-2\g(|n|^{2\alpha}+1)t}\int\Big|\int e^{-i\mathbf{a}(\xi)w\cdot\frac{n}{|n|}}\hat{\chi}_{0}(n,\xi)d\xi\Big|^{2}\frac{e^{-\frac{|w|^{2}}{2|n|^{2}t}}}{\sqrt{2\pi|n|^{2}t}}dwdt\\
 & =\int_{0}^{T}\frac{e^{-2\g(|n|^{2\alpha}+1)t}}{\sqrt{2\pi|n|^{2}t}}\int\Big|e^{-\frac{|w|^{2}}{4|n|^{2}t}}\int e^{-i\mathbf{a}(\xi)\frac{n}{|n|}\cdot w}\hat{\chi}_{0}(n,\xi)d\xi\Big|^{2}dwdt.
\end{align*}
\smallskip

Using  Lemma \ref{lem:integral_estimate} with $a=4|n|^{2}t$, $b(\xi)=\mathbf{a}(\xi)\cdot \frac{n}{|n|}$ and $i(\ve)\sim\ve^{\theta}$ we get
\begin{align*}
\E\int_{0}^{T}|\hat{u}^{0}(n,t)|^{2}dt & \le 2\sqrt{a\pi} \int_{0}^{T}\frac{e^{-2\g(|n|^{2\alpha}+1)t}}{\sqrt{2\pi|n|^{2}t}}\int_{0}^\infty \tau e^{-\tau^{2}}\iota(\frac{2\tau}{\sqrt{a}})d\tau dt\|\hat{\chi}_{0}(n,\cdot)\|_{2}^{2}\\
 & \lesssim a^{\frac{1-\theta}{2}} \int_{0}^{T}\frac{e^{-2\g(|n|^{2\alpha}+1)t}}{\sqrt{2\pi|n|^{2}t}}\int_{0}^\infty\tau ^{1+\theta}e^{-\tau^{2}}d\tau dt\|\hat{\chi}_{0}(n,\cdot)\|_{2}^{2}\\
 & \lesssim \int_{0}^{T}e^{-2\g(|n|^{2\alpha}+1)t}(|n|^{2}t)^{-\frac{\theta}{2}}dt\|\hat{\chi}_{0}(n,\cdot)\|_{2}^{2}\\
 & \lesssim  |n|^{-\theta}\int_{0}^{T}e^{-2\g(|n|^{2\alpha}+1)t}t{}^{-\frac{\theta}{2}}dt\|\hat{\chi}_{0}(n,\cdot)\|_{2}^{2}\\
& \lesssim  |n|^{-\theta}\g^{-\frac{2-\theta}{2}}(|n|^{2\alpha}+1)^{-\frac{2-\theta}{2}} \int_0^\infty e^{-t} t^{-\frac{\theta}{2}} dt \|\hat\chi_{0}(\cdot,\cdot)\|_{2}^{2},
\end{align*}
and, hence, 
\begin{equation}\begin{aligned}\label{eq:u_0-est}
\E\int_{0}^{T}|\hat{u}^{0}(n,t)|^{2}dt  
&\lesssim |n|^{-\theta}\g^{-\frac{2-\theta}{2}}(|n|^{2\alpha}+1)^{-\frac{2-\theta}{2}}\|\hat{\chi}_{0}(n,\cdot)\|_{2}^{2}\\
&\lesssim \gamma^{- \frac{2-\theta}{2}}\|\hat{\chi}_{0}(n,\cdot)\|_{2}^{2}.
\end{aligned}\end{equation}

Combining the previous estimates, after summing over $n$,  we obtain 
\begin{equation}\begin{aligned}\label{takis1001-2}
\E\int_{0}^{T}\|u^{0}(t)\|_{2}^{2}dt
&= \E\int_0^T\|\hat u(t)\|^2 dt 
\lesssim  \gamma^{-\frac{2-\theta}{2}}\|\chi_{0}(\cdot,\cdot)\|_{2}^{2}\\
& \lesssim \gamma^{-\frac{2-\theta}{2}}\|\chi_{0}(\cdot,\cdot)\|_{1}= \gamma^{-\frac{2-\theta}{2}}\|u_{0}\|_{1}.
\end{aligned}\end{equation}

\subsection{The  estimate of  \texorpdfstring{$u^{1}$}{u\^\{1\}} }
Let $\bar \omega_n :=\g(|n|^{2\a}+1)$. For each $n\in \Z^N$, the Fourier transform $\hat{u}^{1}(n,t)$ of $u^{1}(t)$ in $x$ is given by 
\begin{align*}
\hat{u}^{1}(n,t) & = {\bar \omega_n} \int_{0}^{t}\int e^{-i\mathbf{a}(\xi)(\boldsymbol{\b}(t)-\boldsymbol{\b}(s))\cdot n-\g(|n|^{2\alpha}+1)(t-s)}\hat{\chi}(n,\xi,s)d\xi ds\\
 & ={\bar \omega_n} \int_{0}^{t} e^{-\bar \omega_n(t-s)}\int e^{-i\mathbf{a}(\xi)(\boldsymbol{\b}({t})-\boldsymbol{\b}({s}))\cdot n}\hat{\chi}(n,\xi,s)d\xi ds.
\end{align*}

Integrating in $t$, taking expectation and using  that $\int_{0}^{t} {\bar \omega_n} e^{-{\bar \omega_n}r}dr\le1$, we find 
\begin{align*}
 \E\int_{0}^{T}|\hat{u}^{1}|^{2}(n,t)dt
=& \E \int_{0}^{T} \big | \int_{0}^{t} {\bar \omega_n} e^{-{\bar \omega_n}(t-s)} \int e^{-i\mathbf{a}(\xi)(\boldsymbol{\b}(t)-\boldsymbol{\b}(s))\cdot n}\hat{\chi}(n,\xi,s)d\xi ds \big |^{2}dt\\
= & \E \int_{0}^{T} \big | \int_{0}^{t} {\bar \omega_n} e^{-{\bar \omega_n}r} \int e^{-i\mathbf{a}(\xi)(\boldsymbol{\b}({t})-\boldsymbol{\b}({t-r}))\cdot n}\hat{\chi}(n,\xi,t-r)d\xi dr \big | ^{2}dt\\
\le & \E \int_{0}^{T}
\int_{0}^{t} \bar \omega_ne^{-{\bar \omega_n}r}\big |\int e^{-i\mathbf{a}(\xi)(\boldsymbol{\b}({t})-\boldsymbol{\b}({t-r}))\cdot n}\hat{\chi}(n,\xi,t-r)d\xi \big |^{2}drdt.
\end{align*}

In view of the facts that  $\hat{\chi}$ is $\mcF_{t}$-adapted and the increments $\boldsymbol{\b}(t)-\boldsymbol{\b}({t-r})$ are independent of $\mcF_{t-r}$,  using the scaling properties of the Brownian motion we find 
\begin{align*}
 & \E \big | \int e^{-i\mathbf{a}(\xi)(\boldsymbol{\b}({t})-\boldsymbol{\b}({t-r}))\cdot n}\hat{\chi}(n,\xi,t-r)d\xi \big |^{2}\\
 & =\E [\E\big | \int e^{-i\mathbf{a}(\xi)(\boldsymbol{\b}({t})-\boldsymbol{\b}({t-r}))\cdot n}\hat{\chi}(n,\xi,t-r)d\xi\big |^{2}|\mcF_{t-r}]\\
 & =\E \td \E\big | \int e^{-i\mathbf{a}(\xi)(\boldsymbol{\b}({t})-\boldsymbol{\b}({t-r}))(\td \o)\cdot n}\hat{\chi}(n,\xi,t-r)(\o)d\xi\big |^{2}\\
 & =\E \td \E \big | \int e^{-i\mathbf{a}(\xi)\boldsymbol{\b}({|n|^{2}r})(\td \o)\cdot\frac{n}{|n|}}\hat{\chi}(n,\xi,t-r)(\o)d\xi \big |^{2} \\
 & =\E \int\big | \int e^{-i\mathbf{a}(\xi)w\cdot\frac{n}{|n|}}\hat{\chi}(n,\xi,t-r)d\xi \big |^{2}\frac{e^{-\frac{|w|^{2}}{2|n|^{2}r}}}{\sqrt{2\pi|n|^{2}r}}dw\\
 & =\frac{1}{\sqrt{2\pi|n|^{2}r}}\E \int \big |e^{-\frac{|w|^{2}}{4|n|^{2}r}}\int e^{-i\mathbf{a}(\xi)w\cdot\frac{n}{|n|}}\hat{\chi}(n,\xi,t-r)d\xi \big |^{2}dw,
\end{align*}
where $\td \E$ denotes the expectation with respect to $\td \o$.
\smallskip

Employing  again Lemma \ref{lem:integral_estimate} with $a=4|n|^{2}r$ and  $b(\xi)=\mathbf{a}(\xi)\cdot \frac{n}{|n|}$ we get 
\begin{align*}
  &\int\big | e^{-\frac{|w|^{2}}{4|n|^{2}r}}\int e^{-i\mathbf{a}(\xi)w\cdot\frac{n}{|n|}}\hat{\chi}(n,\xi,t-r)d\xi \big |^{2}dw \\
  &\lesssim \sqrt{a}\int_{0}^\infty \tau e^{-\tau^{2}}\iota(\frac{2\tau}{\sqrt{a}})d\tau \|\hat{\chi}(n,\cdot,t-r)\|_{2}^{2}\\[1mm]
 & \lesssim a^{\frac{1-\theta}{2}}\int_{0}^\infty \tau^{\theta+1}e^{-\tau ^{2}}d\tau \|\hat{\chi}(n,\cdot,t-r)\|_{2}^{2}
  \lesssim (|n|^{2}r)^{\frac{1-\theta}{2}}\|\hat{\chi}(n,\cdot,t-r)\|_{2}^{2},
\end{align*}
and, hence,
$$  \E\big | \int e^{-i\mathbf{a}(\xi)(\boldsymbol{\b}({t})-\boldsymbol{\b}({t-r}))\cdot n}\hat{\chi}(n,t-r)d\xi\big |^{2}\\
  \lesssim (|n|^{2}r)^{-\frac{\theta}{2}}\E\|\hat{\chi}(n,\cdot,t-r)\|_{2}^{2}.$$
Combining all the above estimates we find
$$\E\int_{0}^{T}|\hat{u}^{1}|^{2}(n,t)dt 
 \lesssim \int_0^T\int_0^t \bar \omega_ne^{-{\bar \omega_n}r} (|n|^{2}r)^{-\frac{\theta}{2}} \E\|\hat{\chi}(n,\cdot,t-r)\|_{2}^{2}drdt.$$
Young's  inequality then yields
$$\E\int_{0}^{T}|\hat{u}^{1}|^{2}(n,t)dt 
 \lesssim \int_0^T \bar \omega_ne^{-{\bar \omega_n}r} (|n|^{2}r)^{-\frac{\theta}{2}} dr \int_0^T \E\|\hat{\chi}(n,\cdot,r)\|_{2}^{2}dr,$$
and, in view of the fact that, for $\theta\in [0,1]$,
$$
\int_{0}^{T} \bar \o_{n}e^{-\bar \o_{n}r}r{}^{-\frac{\theta}{2}}dr  \le \bar \o_{n}^{1+\frac{\theta}{2}}\int_{\R}e^{- \bar \o_{n}r}(\bar \o_{n}r){}^{-\frac{\theta}{2}}dr 
  =\bar \o_{n}^{\frac{\theta}{2}}\int_{\R}e^{-r}r{}^{-\frac{\theta}{2}}dr <\infty,
$$
we conclude  that, for $n \in \Z^N \setminus \{0\}$,
\begin{equation}\label{eq:u^1-est}
 \E\int_{0}^{T}|\hat{u}^{1}|^{2}(n,t)dt \lesssim \g^{\frac{\theta}{2}}(|n|^{2\a-2}+|n|^{-2})^{\frac{\theta}{2}}\E\|\hat{\chi}(n,\cdot,\cdot)\|_{L^{2}(\R \times [0,T])}^{2},
 \end{equation}
while,  in view of  \eqref{eq:mean-preservation}, 
\begin{align*}
\hat{u}^{1}(0,t) &  =\int_{\TT^N} u^{1}(t,x)dx  =\int_{0}^{t} \int_{{\TT^N} \times \R}\g e^{-\g(t-s)}\hat{\chi}(0,\xi,s)d\xi dxds\\
& =\int_{0}^{t}\g e^{-\g(t-s)}\int_{\TT^N} u(x,s)dxds =0.
\end{align*}
Since 
\begin{align*}
\E\|\hat{\chi}\|_{L^{2}(\Z^N \times \R \times [0,T])}^{2} &=  \sum_{n}\E\int_0^T \int_\R |\hat{\chi}(n,\xi,t)|^{2}d\xi dt =\E\|\chi \|_{L^{2}(\TT^N \times \R \times [0,T])}^{2}\\
& \le  \E\|\chi \|_{L^{1}(\TT^N \times \R \times [0,T])}
=  \E\int_{0}^{T}\|u(t)\|_{1}dt,
\end{align*}
we conclude that
\begin{equation}\label{takis10001}
\E\int_{0}^{T}\|u^{1}(t)\|_{2}^{2}dt  \lesssim \g^{\frac{\theta}{2}}\E\int_{0}^{T}\|u(t)\|_{1}dt.
\end{equation}

\subsection{The estimate of  \texorpdfstring{$Q$}{Q}}
 For $\l\ge 0$ and $\vp \in L^\infty( \Omega \times [0,T]; C^\infty(\TT^N))$ (in what follows, unless necessary, we do not display the dependence of $\vp$ in $\omega$),  let 
\begin{align*}
 &<(-\D)^{\frac{\l}{2}}Q(t),\vp(t)>:= \int_{0}^{t}\int\partial_{\xi}(S_{A_{\gamma}(\xi)}(s,t)((-\D)^{\frac{\l}{2}}\vp(t))(x)dm(x,\xi,s )\\
  & =\int_{0}^{t}\int \boldsymbol a'(\xi)(\boldsymbol\b_{t}-\boldsymbol\b_{s})\cdot D S_{A_{\gamma}(\xi)}^{*}(s,t)((-\D)^{\frac{\l}{2}}\vp(t))(x)dm(x,\xi,s),
\end{align*}
where the second equality is immediate from the definitions of $S_{A_{\gamma}(\xi)}$ and $S^*_{A_{\gamma}(\xi)}$.
\smallskip

To conclude the ongoing proof it is enough to take  $\l=0$, while $\l>0$ is needed for Theorem~\ref{thm:reg}. Since the arguments are similar, we present here the general case. 
\vskip.0125in

We  aim for a bound of  the $L^{1}(\TT^{N} \times \R \times [0,T] \times \Omega)$ norm of $u$ and, hence, we need to estimate 
\begin{align*}
  &\E\int_{0}^{T}\<(-\D)^{\frac{\l}{2}}Q(t),\vp(t)\>dt =\\ &\E\int_{0}^{T}\int_{0}^{t}\int_{{\TT^N} \times \R}   \boldsymbol a'(\xi)(\boldsymbol\b({t})-\boldsymbol\b(s))\cdot D S_{A_{\gamma}(\xi)}^{*}(s,t)((-\D)^{\frac{\l}{2}}\vp(t))(x)dm(x,\xi,s)dt.
\end{align*}
In view of  \cite[Lemma 9]{DV13}, for any  $\psi \in C^\infty(\TT^N)$, 
we have
\begin{align*}
&\|  D(S_{A_{\gamma}(\xi)}^{*}(s,t)(-\D)^{\frac{\l}{2}}\psi)\|_{\infty}=\| D \left(e^{-\g(t-s)B}(-\D)^{\frac{\l}{2}}\psi(\cdot)\right)\left(\cdot-\mathbf{a}(\xi)(\boldsymbol{\b}({t})-\boldsymbol{\b}({s}))\right)\|_{\infty} \\ 
&=\| D \left(e^{-\g(t-s)B}(-\D)^{\frac{\l}{2}}\psi(\cdot)\right)\left(\cdot-\mathbf{a}(\xi)(\boldsymbol{\b}({t})-\boldsymbol{\b}({s}))\right)\|_{\infty}\\
& \le e^{-\g(t-s)}\| D e^{-\g(t-s)(-\D)^\a}(-\D)^{\frac{\l}{2}}\psi\|_{\infty} \lesssim e^{-\g(t-s)}(\gamma(t-s))^{-\frac{\l+1}{2\a}}\|\psi\|_{\infty}. 
\end{align*}
Let
$\mu_{\a,\l}:=\frac{\l+1}{2\a}.$ It follows that, for all  $\vp \in L^\infty( \Omega \times [0,T]; C^\infty(\TT^N)),$ 
\begin{align*}
 & \E\int_{0}^{T}\<(-\D)^{\frac{\l}{2}}Q(t),\vp(t)\>dt \\
 & = \E\int_{0}^{T}\int_{0}^{t}\int_{{\TT^N} \times \R} \partial_{\xi}(S_{A_{\gamma}(\xi)}(s,t)(-\D)^{\frac{\l}{2}}\vp(t))dm(x,\xi,s)dt\\
  & \lesssim \|\vp\|_{\infty}\E\int_{0}^{T}\int_{0}^{t}|\boldsymbol\b({t})-\boldsymbol\b({s})|e^{-\g(t-s)}(\gamma(t-s))^{-\mu_{\a,\l}}\int_{{\TT^N} \times \R} (1+|\xi|^m)dm(x,\xi,s)dt.
\end{align*}

At this point, as in the previous step, we take conditional expectations and use that $\boldsymbol\b({t})-\boldsymbol\b({t-r})$ is independent of $\mathcal F_{t-r}$. The argument is, however, more complicated due to the lack of regularity of $m$ which requires the use of an additional regularization layer.
\smallskip

Let $l$ be the random measure on $[0,T]$ defined by 
$$l([s,t]):=\int_s^t \int_{{\TT^N}\times \R} (1+|\xi|^m)dm(x,\xi,r).$$ Since $m$ is a kinetic measure, we can choose $\mcF_t$-adapted approximations $l^\ve$ such that,  $\P$-a.s.,  $l^\ve \in C([0,T])$ and  $l^\ve \to l$ weak $\star$. Hence, 
\begin{align*}
  &\E\int_{0}^{T}\int_{0}^{t}|\boldsymbol\b({t})-\boldsymbol\b({s})|e^{-\g(t-s)}(\gamma(t-s))^{-\mu_{\a,\l}}\int_{{\TT^N}  \times \R} (1+|\xi|^m)dm(x,\xi,s)dt\\
  &=\E\int_{0}^{T}\int_{0}^{t}|\boldsymbol\b({t})-\boldsymbol\b({s})|e^{-\g(t-s)}(\gamma(t-s))^{-\mu_{\a,\l}}dl(s)dt\\
  &=\lim_{\ve\to 0}\E\int_{0}^{T}\int_{0}^{t}|\boldsymbol\b({t})-\boldsymbol\b({s})|e^{-\g(t-s)}(\gamma(t-s))^{-\mu_{\a,\l}}l^\ve(s) dsdt.
\end{align*}
Using the independence of $\boldsymbol\b({t})-\boldsymbol\b({s})$ from $\mcF_s$ and the $\mcF_s$-measurability of $l^\ve$ we find
\begin{align*}
  \E|\boldsymbol\b({t})-\boldsymbol\b({s})|l^\ve(s)
   &=\E[\E|\boldsymbol\b({t})-\boldsymbol\b({s})|l^\ve(s)|\mcF_s] 
=\E|\boldsymbol\b({t})-\boldsymbol\b({s})|\E l^\ve(s) \\
   &= (t-s)^{\frac{1}{2}} \E l^\ve(s).
\end{align*}
Employing again Young's inequality we obtain
 \begin{align*}
   &\E\int_{0}^{T}\int_{0}^{t}|\boldsymbol\b({t})-\boldsymbol\b({s})|e^{-\g(t-s)}(\gamma(t-s))^{-\mu_{\a,\l}}\int_{{\TT^N} \times \R} (1+|\xi|^m)dm(x,\xi,s)dt\\
   &=\lim_{\ve\to 0}\int_{0}^{T}\int_{0}^{t} (t-s)^{\frac{1}{2}} e^{-\g(t-s)}(\gamma(t-s))^{-\mu_{\a,\l}}\E l^\ve(s) dsdt \\
   &\le \lim_{\ve\to 0}\int_{0}^{T}t^{\frac{1}{2}} e^{-\g t}(\gamma t)^{-\mu_{\a,\l}}dt \int_0^T\E l^\ve(t) dt \\
   &\le\int_{0}^{T} t^{\frac{1}{2}}e^{-\g t}(\gamma t)^{-\mu_{\a,\l}}dt \E\int_0^T \int_{{\TT^N}\times \R} (1+|\xi|^m)dm(x,\xi,t).
 \end{align*}
In conclusion,
\begin{align*}
&\E\int_{0}^{T}\<(-\D)^{\frac{\l}{2}}Q(t),\vp(t)\>dt \\
& \lesssim \|\vp\|_{\infty}\gamma^{-\mu_{\a,\l}}\int_{0}^{T}e^{-\g t}t{}^{\frac{1}{2}-\mu_{\a,\l}}dt\E\int_{0}^{T}\int_{{\TT^N}\times \R} (1+|\xi|^m)dm(x,\xi,s).
\end{align*}

Moreover note that, if $\mu_{\a,\l}<\frac{3}{2}$, then
\begin{align*}
\int_{0}^{T}t{}^{\frac{1}{2}-\mu_{\a,\l}}e^{-\g t}dt  &=\g^{-\frac{1}{2}+\mu_{\a,\l}}\int_{0}^{T}(\g t){}^{\frac{1}{2}-\mu_{\a,\l}}e^{-\g t}dt =\g^{-\frac{3}{2}+\mu_{\a,\l}}\int_{\R_{+}}t{}^{\frac{1}{2}-\mu_{\a,\l}}e^{-t}dt \\
&\le C \g^{-\frac{3}{2}+\mu_{\a,\l}}.
\end{align*}
We use next Lemma \ref{lem:m-bound} and, in view of all the above,  we get 
\begin{align}\label{eq:Q_est}
  & \E\int_{0}^{T} \<(-\D)^{\frac{\l}{2}} Q(t),\vp(t)\>dt
   \lesssim  \|\vp\|_{\infty}\g^{-\frac{3}{2}}(\|u_{0}\|_{2}^{2}+\|u_{0}\|_{m+2}^{m+2}).
\end{align}
As mentioned earlier, for an estimate on the energy decay it is enough to consider $\l=0$ in which case $\mu_{\a,0}  =\frac{1}{2\a}$ and, assuming $\a>\frac{1}{3}$, we obtain the estimate
\begin{equation}\label{takis100001}
\E\int_{0}^{T} \<Q(t),\vp(t)\>dt
   \lesssim \g^{-\frac{3}{2}}\|\vp\|_{\infty}(\|u_{0}\|_{2}^{2}+\|u_{0}\|_{m+2}^{m+2}).
 \end{equation}

\subsection{The proof of Theorem \ref{main}}
Since, for all  $\vp \in L^\infty( \Omega \times [0,T]; C^\infty(\TT^N))$, 
 $$\<u(t),\vp\> = \<u^{0}(t),\vp\>+\<u^{1}(t),\vp\>+\<Q(t),\vp\>,$$
we find 
\begin{align*}
  &\E\int_0^T \<u(t),\vp(t)\> dt\\
  &\le \|\vp\|_\infty \E\int_0^T [\|u^{0}(t)\|_1 + \|u^{1}(t)\|_1]dt + \E\int_0^T \<Q(t),\vp(t)\>dt\\
  &\le 2 \|\vp\|_\infty T^{\frac{1}{2}}\left(\E\int_{0}^{T}[\|u^{0}(t)\|_{2}^{2}+\|u^{1}(t)\|_{2}^{2}]dt\right)^{\frac{1}{2}} + \E\int_0^T \<Q(t),\vp(t)\>dt.
\end{align*}
The inequality above as well as \eqref{takis1001-2}, \eqref{takis10001} and \eqref{takis100001}  and  the Cauchy-Schwartz  inequality yield, for some $C>0$,  the following sequence of inequalities. 
\begin{align*}
 & \E\int_0^T\<u(t),\vp(t)\> dt\\
 &  \le C \|\vp\|_\infty T^{\frac{1}{2}}({\gamma^{-\frac{2-\theta}{2}}}\|u_{0}\|_{1}+\g^{\frac{\t}{2}}\E\int_{0}^{T}\|u(t)\|_{1}dt)^{\frac{1}{2}}+\g^{-\frac{3}{2}} (\|u_{0}\|_{2}^{2}+\|u_{0}\|_{m+2}^{m+2})\\
& \le C \|\vp\|_\infty T^{\frac{1}{2}}({\gamma^{-\frac{2-\theta}{2}}}\|u_{0}\|_{1}+ 4TC^2\g^{\theta}+\frac{1}{4TC^2}(\E\int_{0}^{T}\|u(t)\|_{1}dt)^{2})^{\frac{1}{2}} 
 \\&\qquad\qquad\qquad+\g^{-\frac{3}{2}}(\|u_{0}\|_{2}^{2}+\|u_{0}\|_{m+2}^{m+2})\\
 & \le   \|\vp\|_\infty
  [C(T^{\frac{1}{2}}{\gamma^{-\frac{2-\theta}{4}}}\|u_{0}\|_{1}^{\frac{1}{2}}+T\g^{\frac{\theta}{2}}+\g^{-\frac{3}{2}}(\|u_{0}\|_{2}^{2}+\|u_{0}\|_{m+2}^{m+2})) +\frac{1}{2}\E\int_{0}^{T}\|u(t)\|_{1}dt],
 \end{align*}
Hence, 
\begin{align*}
 &\E\int_{0}^{T}\|u(t)\|_{1}dt \\
 &\le C(T^{\frac{1}{2}}{\gamma^{-\frac{2-\theta}{4}}}\|u_{0}\|_{1}^{\frac{1}{2}}+T\g^{\frac{\theta}{2}}+\g^{-\frac{3}{2}}(\|u_{0}\|_{2}^{2}+\|u_{0}\|_{m+2}^{m+2})) +\frac{1}{2}\E\int_{0}^{T}\|u(t)\|_{1}dt
\end{align*}
and, thus,
\begin{align*}
\E\frac{1}{T}\int_{0}^{T}\|u(t)\|_{1}dt & \lesssim T^{-\frac{1}{2}}{\gamma^{-\frac{2-\t}{4}}}\|u_{0}\|_{1}^{\frac{1}{2}}+\g^{\frac{\t}{2}}+\g^{-\frac{3}{2}}T^{-1}(\|u_{0}\|_{2}^{2}+\|u_{0}\|_{m+2}^{m+2}),
\end{align*}
and, since  $t\mapsto\|u(t)\|_{1}$ in non-increasing, 
\begin{align*}
\E\|u(T)\|_{1} & \le\E\frac{1}{T}\int_{0}^{T}\|u(t)\|_{1}dt\\
 & \lesssim{T^{-\frac{1}{2}}}{\gamma^{-\frac{2-\theta}{4}}}\|u_{0}\|_{1}^{\frac{1}{2}}+\g^{\frac{\theta}{2}}+\g^{-\frac{3}{2}}T^{-1}(\|u_{0}\|_{2}^{2}+\|u_{0}\|_{m+2}^{m+2}).
\end{align*}
Choosing $\g =T^{-a}$ we get 
\begin{align*}
\E\|u(T)\|_{1} & \lesssim T^{-\frac{1}{2}+a(\frac{2-\theta}{4})}\|u_{0}\|_{1}^{\frac{1}{2}}+2T^{-\frac{a\theta}{2}}+
T^{\frac{3a}{2}-1}(\|u_{0}\|_{2}^{2}+\|u_{0}\|_{m+2}^{m+2}).
\end{align*}
Finally, we  let $a=\frac{2}{3+\t}$ and obtain, assuming that  $T\ge1$,
\begin{align*}
\E\|u(T)\|_{1} & \lesssim T^{-\frac{\theta}{3+\theta}}\left(\|u_{0}\|_{1}^{\frac{1}{2}}+1+\|u_{0}\|_{2}^{2}+\|u_{0}\|_{m+2}^{m+2}\right);
\end{align*}
note that the rate is independent of the choice of $\a$.  
\smallskip

It follows that  $\d_{0}$ is the unique invariant measure and, since $t\mapsto\|u(t)\|_{1}$ is pathwise non-increasing, we have that a.s.
\begin{equation}\label{takis1000}
\lim_{t\to \infty}  \|u(t)\|_{1}=0.
\end{equation} 

\subsection{The random attractor} Since $L^{1}(\Omega)$ is separable, there exist a countable set $M \subset L^1(\Omega)\cap L^2(\Omega)$, dense in $ L^1(\Omega)$ and  $\O_{0}\subseteq\O$ of full $\P$-measure, such that,  for all $u_{0}\in M$ and all $\o\in\O_{0}$, as $t\to\infty$,
\[
\|\vp(t,\o)u_{0}\|_{1}=\|u^{u_{0}}(t,\o)\|_{1}\to0. 
\]

Fix $\ve>0$ and a  compact $K\subset L^{1}(\Omega) $. There exist a positive integer $L=L(K,\ve)$ and  $\{x_{i}\}_{i=1}^{L} \subset M$  such that $K \subset \Cup_{i=1}^L B(x_i, \frac{1}{2}\ve)$. It follows that, for all $\o\in\O_{0}$,  there exists $t_0(\ve) >0$ such that, if $t\ge t_{0}(\ve)$,
$$\|\vp(t,\o)x_{i}\|_{1}\le\frac{\ve}{2}.$$ The pathwise contraction property  in $L^{1}(\Omega)$ then yields  that  $\vp(t,\o)(K)$ has a cover by balls of radius $\frac{\ve}{2}$ centered at $\{\vp(t,\o)x_{i}\}_{i=1}^{L}$ and, thus, for all $\o\in\O_{0}$   and   $t\ge t_{0}(\ve)$, $$\sup_{x\in K}\|\vp(t,\o)x\|_{1}\le\ve.$$ Hence, $0$ is a forward random attractor.
 \smallskip
 
Moreover, in view of the $\P$-invariance of $\t_t$, $\sup_{x\in K}\|\vp(t,\t_{-t}\o)x\|_{1}\to0$ in probability.  Since 
 $t\mapsto\|\vp(t,\t_{-t}\o)x\|_{1}$ is pathwise non-increasing, the above observation also implies, that, $\P$-a.s., 
$$\sup_{x\in K}\|\vp(t,\t_{-t}\o)x\|_{1}\to0$$
and, thus, $0$ is also a pullback random attractor.

\section{Rate of convergence for deterministic SCL.}\label{sec:det_case}
The techniques used in the proof of Theorem \ref{main} may also be employed  in the deterministic setting. In fact, the arguments are somewhat simpler and the estimates can be taken from \cite{DV13} since no difficulties arise  when the flux is not stochastic. 
\begin{proof}[Proof of Theorem \ref{prop:det}]
We follow the proof Theorem \ref{main} noting that instead of applying  Lemma \ref{lem:integral_estimate} we use \cite[Lemma 2.4]{BD99} which yields the estimates:
\begin{align*}
\ \   \int _0^T \|u^0(t)\|_{H^{\a+(\frac{1}{2}-\a)\t}}dt& \lesssim \g^{\t-1} \|u_0\|_{L^1}dt,\\
  \int _0^T \|u^1(t)\|_{H^{(\frac{1}{2}-\a)\t}}dt& \lesssim \g^\t \int_0^T \|u^1(t)\|_{L^1}dt,
  \end{align*}
and   
\[  \|Q\|_{L^1([0,T];W^{\l,1})} \lesssim \g^{-2} (\|u_{0}\|_{2}^{2}+\|u_{0}\|_{m+2}^{m+2}).\]
Proceeding as before we get 
\begin{align*}
  \|u(T)\|_{1}  \lesssim{T^{-\frac{1}{2}}}{\gamma^{-\frac{1-\theta}{2}}}\|u_{0}\|_{1}^{\frac{1}{2}}+\g^{\theta}+\g^{-2}T^{-1}(\|u_{0}\|_{2}^{2}+\|u_{0}\|_{m+2}^{m+2}).
\end{align*}
Choosing $\g =T^{-a}$ we find 
\begin{align*}
 \|u(T)\|_{1} & \lesssim T^{-\frac{1}{2}+a(\frac{1-\theta}{2})}\|u_{0}\|_{1}^{\frac{1}{2}}+2T^{-a\theta}+T^{2a-1}(\|u_{0}\|_{2}^{2}+\|u_{0}\|_{m+2}^{m+2}),
\end{align*}
which, for $a=\frac{1}{\t+2}$ and  for all $T\ge 1$, yields
\begin{align*}
 \|u(T)\|_{1} & \lesssim T^{-\frac{\t}{\t+2}}(1 +\|u_{0}\|_{m+2}^{m+2}).
\end{align*}
\end{proof}

\section{Regularity of pathwise entropy solutions.}\label{sec4}

In this section we present the proof of Theorem \ref{thm:reg}.  Since the argument is long, we divide it into three steps. In the first step, we obtain a bound, which we use in the second step to bootstrap and, hence, to improve the estimate. The conclusion is then shown in the last step.

\begin{proof}[Proof of Theorem \ref{thm:reg}]  \textit{Step 1:} We assume that, for some $\tau\in[0,1],$ $\chi = \chi(u) \in L^2(\R\times[0,T]\times\Omega;H^\tau(\R^N)),$  and note that
  \eqref{eq:u_0-est} with $\g=1$ yields 
\begin{align*}
\E\int_{0}^{T}\|u^{0}(t)\|_{H^{\t+\a (2-\t)}}^{2}dt & \lesssim \|u_0\|_1.
\end{align*}

Moreover, multiplying \eqref{eq:u^1-est}  by $|n|^\tau$ and taking  $\g=1$, we obtain
\begin{align*}
\E\int_{0}^{T}\|u^{1}(t)\|_{H^{\t(1-\a)+\tau}}^{2}dt & \lesssim \|\chi \|_{L^2(\R\times[0,T]\times\O;H^\tau(\R^N))}^2 .
\end{align*}
As in the last section, for all  $\vp \in L^\infty([0,T]\times\Omega; C^\infty(\TT^N))$, we have
\begin{align*}
  &\E\int_0^T \<(-\D)^\frac{\l}{2} u(t),\vp(t)\> dt\\
  &= \E\int_0^T\<(-\D)^\frac{\l}{2}u^{0}(t),\vp(t)\>+\<(-\D)^\frac{\l}{2}u^{1}(t),\vp(t)\>dt+\E\int_0^T\<(-\D)^\frac{\l}{2} Q(t),\vp(t)\>dt\\
  &\le \|\vp\|_{\infty} \E\int_0^T [\|u^{0}(t)\|_{W^{\l,1}} + \|u^{1}(t)\|_{W^{\l,1}}]dt + \E\int_0^T \<(-\D)^\frac{\l}{2} Q(t),\vp(t)\>dt\\
  &\le \|\vp\|_{\infty}T^\frac{1}{2}( \E\int_0^T[ \|u^{0}(t)\|_{{W^{\l,1}}}^2 + \|u^{1}(t)\|_{W^{\l,1}}^2]dt)^\frac{1}{2} + \E\int_0^T \<(-\D)^\frac{\l}{2} Q(t),\vp(t)\>dt.
\end{align*}
We assume for the moment that, for  $\l\in (0,1)$,
\begin{equation}\label{eq:assumed_emb}
   H^{\t+\a (2-\t)}, H^{\t(1-\a)+\tau}\hookrightarrow W^{\l,1},
\end{equation} 
where $\hookrightarrow$ denotes continuous embedding, and use that $t \to \|u(t)\|_1$ is non-increasing in $t$, the above estimates and  \eqref{eq:Q_est} with $\g=1$, to get  
\begin{equation}\begin{aligned}\label{eq:Q-est-1}
  &\E\int_0^T \<(-\D)^\frac{\l}{2} u(t),\vp(t)\> dt \\
  &\lesssim \|\vp\|_{\infty}(T^\frac{1}{2}( \|u_0\|_1 +\|\chi \|_{L^2(\R\times[0,T]\times\O;H^\tau(\R^N))}^2)^\frac{1}{2} 
  + \|u_{0}\|_{2}^{2}+\|u_{0}\|_{m+2}^{m+2})\\
  &\lesssim  \|\vp\|_{\infty}( T^\frac{1}{2}(1+T^\frac{1}{2})\|u_0\|_1^\frac{1}{2} + \|\chi \|_{L^2(\R\times[0,T]\times\O;H^\tau(\R^N))}+\|u_{0}\|_{m+2}^{m+2}) \\
  &\lesssim 
 \|\vp\|_{\infty}( 1 + \|\chi \|_{L^2(\R\times[0,T]\times\O;H^\tau(\R^N))}^2+\|u_{0}\|_{2+m}^{2+m}),
\end{aligned}\end{equation}
as long as
\begin{equation}\label{eq:mu_reg_cdt}
   \mu_{\a,\l}=\frac{\l+1}{2\a}<\frac{3}{2}.
\end{equation}
Next we take  $u_0 \in BV$ and recall that, in view of  Proposition \ref{prop:well-posedness}, $u(t)\in BV$  for all $t>0$,  Then \eqref{eq:Q-est-1} implies
\begin{equation}\label{eq:reg_prop_1}
  \E\int_0^T \|u(t)\|_{W^{\l,1}} dt
  \le C( 1 + \|\chi \|_{L^2(\R\times[0,T]\times\O;H^\tau(\R^N))}^2 + \|u_{0}\|_{2+m}^{2+m}).
\end{equation}
The general case $u_0 \in L^{2+m}(\TT^N)$  follows easily using approximations and the continuity of the solutions with respect to the initial condition.

From \eqref{eq:reg_prop_1} it follows 
\begin{equation}\label{eq:reg_prop_chi_1} 
   \|\chi\|_{L^1(\R\times[0,T]\times\O;W^{\l,1}(\R^N))}
     \le C( 1 + \|\chi \|_{L^2(\R\times[0,T]\times\O;H^\tau(\R^N))}^2 + \|u_{0}\|_{2+m}^{2+m}).
\end{equation}
Interpolating \eqref{eq:reg_prop_1}  with the obvious bound $\|\chi\|_{L^\infty(\R\times[0,T]\times\O\times\R^N)}\le 1$ yields
\begin{equation}\label{eq:reg_prop_chi}  \|\chi\|_{L^2(\R\times[0,T]\times\O;W^{\frac{\l}{2},2}(\R^N))}^2
     \le C( 1 + \|\chi \|_{L^2(\R\times[0,T]\times\O;H^\tau(\R^N))}^2 + \|u_{0}\|_{2+m}^{2+m}).
\end{equation}
\smallskip

It remains to choose $\l$ and $\a$ to justify the above calculations. To this end, we first note that \eqref{eq:mu_reg_cdt} is equivalent to
 $\l<3\a-1,$
and noting that $W^{r,p}\hookrightarrow W^{s,q}$ if $r\ge s$ and $\frac{r}{N}-\frac{1}{p} \ge \frac{s}{N}-\frac{1}{q} $, we further require that  
$  \t(1-\a)+\tau\ge \l,$
which gives 
\begin{align*}
\frac{\t(1-\a)+\tau}{N}-\frac{1}{2} \ge \frac{\l}{N}-1,
\end{align*}
and, hence, the requirement that 
  $$ \l <(3\a-1) \wedge (\t(1-\a)+\tau).$$
Maximizing the right hand side yields $\a=\frac{\t+1+\tau}{\t+3} \in (0,1)$ and we obtain
\begin{equation}\label{eq:lambda_reg}
  \l < 3\a-1=\frac{2\t}{\t+3}+\frac{3}{\t+3}\tau .
\end{equation}
We note that, since $\tau \le 1$, we have  $2\a \ge \tau$, and, hence, 
\begin{align*}
  \frac{\t+\a (2-\t)}{N}-\frac{1}{2} \ge 
\frac{\t(1-\a)+\tau}{N}-\frac{1}{2}.
\end{align*}
It follows now that $H^{\t+\a (2-\t)} \hookrightarrow W^{\l,1}$.
\smallskip

\textit{Step 2:} Bootstrapping

We bootstrap the arguments in the first step. Using \eqref{eq:reg_prop_chi} with $\tau=0$ and noting that 
$$\|\chi \|_{L^2(\R\times[0,T]\times\O;L^2(\R^N))}\le \|\chi \|_{L^1(\R\times[0,T]\times\O;L^1(\R^N))} \lesssim \|u_0\|_1$$
gives, for 
$\l < \l_0 := \frac{2\t}{\t+3},$ 
$$
\|\chi\|_{L^2(\R\times[0,T]\times\O;W^{\frac{\l}{2},2}(\R^N))}
     \lesssim( 1 + \|u_{0}\|_{2+m}^{2+m}).
$$

Next we iterate this argument and get, from \eqref{eq:reg_prop_chi_1}, that for 
$$\l < \l_n := \frac{2\t}{\t+3}+\frac{3}{\t+3}\frac{\lambda_{n-1}}{2},$$
and some constant $C_n$, 
$$
\|\chi\|_{L^1(\R\times[0,T]\times\O;W^{\l,1}(\R^N))}
     \le C_n( 1 + \|u_{0}\|_{2+m}^{2+m}).
$$
Since, as $n\to \infty$, 
$$\l_n \uparrow \l_* = \frac{4\t}{2\t+3},$$
using \eqref{eq:reg_prop_1}  we obtain that, for any 
$$ \l < \l_* = \frac{4\t}{2\t+3},$$
there is some $C>0$ such that
\begin{equation}\label{eq:reg_prop}
 \E\int_0^T \|u(t)\|_{W^{\l,1}} dt
    \le C( 1 + \|u_{0}\|_{2+m}^{2+m}).
\end{equation}

\textit{Step 3:} Conclusion

In view of  \eqref{eq:reg_prop}, for each $\d>0,$ there is a $\d_0 \in [0,\d]$ such that  $\E \|u(\d_0)\|_{W^{\l,1}} < \infty.$ The spatial homogeneity and the contraction property with respect to the initial condition imply that $t\mapsto \|u(t)\|_{W^{\l,1}}$ is non-increasing, and, hence,
\begin{align*}
\sup_{t \ge \d}   \E\|u(t)\|_{W^{\l,1}} < \infty.
\end{align*}
\end{proof}

\section{Pathwise quasi-solutions and regularization by noise.}\label{sec:regularization}

\subsection {Deterministic background}  We recall  (see \cite{DLOW03,P13}) that $u\in C([0,T];L^{1}(\TT))\cap L^{\infty}(\TT \times [0,T])$ is 
a quasi-solution to the deterministic Burgers' equation \eqref{eq:det_scl} if, for every convex entropy-entropy flux pair $(\eta,q)$,
 $-\mu_{\eta}:=\partial_{t}\eta(u)+\partial_{x}q(u)$ is a Radon measure.  Note that the difference between quasi- and entropy solutions is that for the latter $\mu_{\eta}$ is,
 in addition, non-negative.
\smallskip

Moreover (see, for example, Benilan and Kruzkov \cite{BK96}), $u\in C([0,T];L^{1}(\TT))\cap L^{\infty}(\TT \times [0,T])$ is an entropy sub-solution (resp. super-solution) to  \eqref{eq:det_scl}, if, for every convex entropy-entropy flux pair $(\eta,q)$ with $\eta$ nondecreasing (resp. nonincreasing), $-\mu_{\eta}:=\partial_{t}\eta(u)+\partial_{x}q(u)$ is a nonnegative Radon measure.
\smallskip

The next claim is about the existence of quasi-solutions to \eqref{eq:det_scl}.

\begin{lem}\label{lem:takis10001} (i) Any entropy sub- or super-solution to \eqref{eq:det_scl} is also a quasi-solution.

(ii) An entropy solution to  $\partial_t u+\partial_{x}u^{2}=g$,  with source $g \in L^1_{loc}(\TT)$, is a quasi-solution to  \eqref{eq:det_scl}.
\end{lem}
\begin{proof}
The second assertion is proved in \cite{DLOW03, P13}.  For the first claim we consider here the case of $u$ being an entropy sub-solution, since the argument can be easily modified for super-solutions. 

\vskip.075in
Let $(\eta,q)$  be a convex entropy-entropy flux pair. In order to use the sub-solution property of $u$,  we show that it is possible to write $(\eta,q)$ as the difference of two convex entropy-entropy flux pairs $(\eta_1,q_1)$ and $(\eta_2,q_2)$ with  $\eta_1$ and  $\eta_2$ nonincreasing. We also remark that, since $u\in L^\infty(\TT \times [0,T]),$ it is enough to obtain this decomposition  in $[-\|u\|_\infty, \|u\|_\infty]$.
\smallskip

Fix $s_0:=-(\|u\|_\infty +1)$.
Since $\eta$ is convex, there exists some $p_0 \in \R$ such that, for all $s\in \R$,
 $$ \eta(s) \ge \eta(s_0) + p_0 (s-s_0).$$
 
 Let $\eta_1,\eta_2:\R \to \R$ be the convex and nondecreasing functions given by
 $$
 \eta_{1}(s):=\begin{cases}
      \eta(s)+|p_0|(s-s_0) \ \text{if } \ s\ge s_0,\\[1mm]
      \eta(s_0) \  \text{if} \  s< s_0,
\end{cases} 
\ \text{and} \ \  \eta_2(s):=|p_{0}|(s-s_0).
$$
It follows that   $$ \eta=\eta_{1}-\eta_{2}  \ \text{ and } q=q_{1}-q_{2}  \ \   \text{in } \  [s_0,\infty).$$

   Then  $-\mu_{\eta}:=\partial_{t}\eta(u)+\partial_{x}q(u)$ is a Radon measure, since  
\begin{align*}
  -\mu_{\eta}
  &:=\partial_{t}\eta(u)+\partial_{x}q(u)  =\partial_{t}\eta_{1}(u)+\partial_{x}q_{1}(u)-(\partial_{t}\eta_{2}(u)+\partial_{x}q_{2}(u)) \\
  &=(-\mu_{\eta_{1}})-(-\mu_{\eta_{2}}),
\end{align*}
 where $-\mu_{\eta_{1}},-\mu_{\eta_{2}}$ are the nonnegative Radon measures coming from the sub-solution property of $u$ with entropies $\eta_{1}, \eta_{2}$ respectively. 
\end{proof}

Although in general quasi-solutions need not be weak solutions,  \cite{DLW03} provides an explicit example of a quasi-solution  to \eqref{eq:det_scl} that is a weak but not an entropy solution. We remark that in \cite{DLW03}, which  is about the Cauchy problem in $\R$, is not required for  quasi-solutions to be continuous in $L^{1}(\TT)$. It follows, however, from the explicit construction therein, that the same argument also works in the periodic setting and yields a quasi-solution $u\in C([0,T];L^{1}(\TT))$ to \eqref{eq:det_scl}  with the above properties.
\smallskip

The kinetic formulation of a quasi-solution $u$ to \eqref{eq:det_scl}  (see \cite{DLOW03}) is that the $\chi$ defined as in \eqref{eq:char_fctn} satisfies,  in the sense of distributions,
\begin{equation}\label{eq:kinetic-quasi}
  \partial_{t}\chi+2\xi\partial_{x}\chi =\partial_{\xi}m,
\end{equation}
where $m$  is a Radon measure on $\TT\times\R\times [0,\infty)$,  which is supported in $[-\|u\|_{\infty},\|u\|_{\infty}]$ in the $\xi$-variable and has finite total variation $|m|$ in $\TT\times\R\times[0,T]$, for each $T>0$. 
\vskip.0125in

Once there is  a kinetic formulation, it is immediate that quasi-solutions can also be described using the convolution along characteristics, as it was done in Section \ref{sec:defn_pathwise_entropy} for entropy solutions, that is to assume \eqref{takis1} with $m$ as above.
\smallskip

Finally, it is straightforward that all the previous statements about quasi-, sub- and super-solutions extend to the time-dependent deterministic Burger's equation 
$$ \partial_u+ \partial_{x}u^2 \dot{\b}=0 \ \text{ in } \ \TT\times(0,\infty),$$
for any   smooth path $\b$.
\smallskip

\subsection{Pathwise quasi-solutions} Using the convolution by characteristics formulation, we introduce now  the notion of quasi-solutions to \eqref{eq:det_scl} to the stochastic Burger's equation \eqref{eq:stoch_burgers}.  For the definition we consider kinetic Radon measures, which have the same measurability properties as kinetic measures used in the definition of pathwise entropy solutions but need not be nonnegative.

\begin{defn}\label{def:path_e-soln-1} A map $u:\TT^N \times [0,\infty) \times \Omega \to \R$  such that, for all $T>0$,  $u\in  (L^{1}\cap L^\infty)(\TT^N \times [0,T]\times \Omega)$, and $\P$-a.s.\ in $\omega$, $u(\cdot, \omega)  \in C \big([0,T]; L^1(\TT^N)\big)\cap L^\infty (\TT^N \times [0,T])$, and $t\mapsto u(t, \cdot )$ is an $\mcF_{t}$-adapted  process in $L^1(\TT^N)$, is a pathwise quasi-solution to \eqref{eq:stoch_burgers}, if there exists a kinetic Radon measure $m$ such that, for all $y\in\TT^N$ and $\vr(x,y,\xi,t)=\rho_0(x-y+ \xi (\beta(t)-\beta(s)))$ with $\vr_{0}\in C^{\infty}(\TT^N)$ and all  $\vp\in C_{c}^{\infty}([0,T)),\psi\in C_c^\infty(\R)$, 
\begin{equation}\begin{split}\label{eq:quasi-stoch}
 & \int_{0}^{T}\int_{\R}\partial_{t}\vp(r)\psi(\xi)(\vr\ast\chi)(y,\xi,r)drd\xi+\int_{\R}\vp(0)\psi(\xi)(\vr\ast\chi)(y,\xi,0)d\xi\\
 & =\int_{0}^{T}\int_{{\TT^N}\times \R}\vp(r)\partial_{\xi}\left(\psi(\xi)\vr(x,y,\xi,r)\right)dm(x,\xi,r).
\end{split}\end{equation}
\end{defn}
\smallskip 

Next we present an example of a pathwise quasi-solution which is not a pathwise entropy solution.  Its construction is based on the observation in Lemma~\ref{lem:takis10001}(i) that, in the deterministic setting, entropy sub- and super-solutions are quasi-solutions. 
\smallskip

Fix a non-negative $u_0\in L^\infty(\TT)$ and a family $(\b^\ve)_{\ve>0}$ of  smooth paths  approximating the Brownian motion $\b$, that is, for $\ve>0$, $\b^\ve  \in C^\infty([0,\infty))$ and, as $\ve \to 0$, $\b^\ve \to  \b$ in $C([0,T])$ for all $T\ge 0$, and let $u^\ve$ be  the nonnegative entropy solution to
 \begin{equation*}
 \begin{cases}
 \partial_tu^\ve+ \partial_{x}(u^\ve)^2\dot{\b}^\ve=u^\ve  \ \text{ in } \ \TT\times(0,\infty),\\[1mm]
 u^\ve=u_0 \  \text{on} \  \TT \times \{0\}.
 \end{cases}
 \end{equation*}
Since $u^\ve$ is an entropy super-solution to 
 \begin{equation*}
 \begin{cases}
 \partial_tu^\ve+  \partial_{x}(u^\ve)^2\dot{\b}^\ve=0 \ \text{ in } \ \TT\times(0,\infty),\\[1mm]
 u^\ve=u_0 \  \text{on}  \  \TT \times \{0\},
 \end{cases}
 \end{equation*}
in view of  Lemma~\ref{lem:takis10001}, $u^\ve$ is also a quasi-solution  with dissipation measure $m^\ve$ having a non-vanishing negative part given by $\chi(u^\ve,\xi)u^\ve$. 
\smallskip

It follows from the density arguments in \cite{LPS13}, that  the  $u^\ve$'s converge in $C([0,T];L^1(\TT))$ to the pathwise entropy solution to
 \begin{equation*}
 \begin{cases}
 du + \partial_{x}u^2\circ d \b=u dt \ \text{ in } \ \TT\times(0,\infty),\\[1mm]
 u =u_0 \  \text{on} \  \TT \times \{0\}.
 \end{cases}
 \end{equation*}

Moreover, since the dissipation measures $m^\ve$'s  have finite total variation, we may extract subsequences $u^{\ve_n} \to u$ and  $m^{\ve_n} \rightharpoonup^* m$. Writing \eqref{eq:quasi-stoch} for $u^{\ve_n}$ and $m^{\ve_n}$ and then taking the limit yields that $u$ is a quasi-solution to \eqref{eq:stoch_burgers}, with dissipation measure $m$ having non-vanishing negative part $\chi(u,\xi)u$.
\smallskip

We remark that it can be shown that pathwise quasi-solutions do not satisfy, in general, the density property of the pathwise entropy solutions. Indeed quasi-solutions to 
equations with smooth paths do not converge, in general, to pathwise quasi-solutions. 

\subsection{The proof of Theorem \ref{thm:regularization}} We  present now the  
\begin{proof}[Proof of Theorem \ref{thm:regularization}]
We first note that the derivation of \eqref{eq:splitup} given in Appendix \ref{app:splitup} does not rely on the assumption that $m$ is a nonnegative measure. Hence, \eqref{eq:splitup} also holds for pathwise quasi-solutions. The proof now is the same as the one for Theorem \ref{thm:reg}, with the exception that in the estimation of $Q$ we cannot use  Lemma \ref{lem:m-bound}, since it is not satisfied for quasi-solutions. Instead, we observe (note that here we have $A(u)=u^{2}$  and $\t=1$) that 
\begin{align}
 & \E\int_{0}^{T}\<(-\D)^{\frac{\l}{2}}Q(t),\vp(t)\>dt\lesssim\|\vp\|_{\infty}\g^{-\frac{3}{2}}\E\int_{0}^{T}\int_{{\TT^{N}}\times\R}d|m|(x,\xi,s). \label{eq:Q_est-1}
\end{align}
It then follows, as before, that 
\[
 \E\int_{0}^{T}\<(-\D)^{\frac{\l}{2}}u(t),\vp(t)\>dt
 \lesssim\|\vp\|_{\infty}\left(1+\|u_{0}\|_{2}^{2}+\E|m|([0,T]\times\TT\times\R)\right),
\]
which implies the claim.
\end{proof}

        

\appendix

\section{Random dynamical systems and random attractors}\label{app:RDS}

We briefly recall some of the concepts used earlier in the proofs and  we refer, for example,  to Arnold \cite{A98}, Crauel and Flandoli  \cite{CF94} and Ochs \cite{O99}, for a comprehensive treatment. 
\smallskip 

Let $(E,d)$ and $\left(\Omega,\mcF,\P,\t\right)$ be respectively a complete, separable metric space and a metric dynamical system, where  $(\Omega,\mcF,\P)$ is,  a not necessarily complete, probability space  $(\Omega,\mcF,\P)$  and $\t:=\left(\theta_{t}\right)_{t\in\R}$ a group of jointly measurable maps on $\left(\Omega,\mcF,\P\right)$ which leaves $\P$ invariant.
\smallskip

A random dynamical system is a measurable map $\vp:[0,\infty) \times\O\times E\to E$ such that, for all  $ \o\in\O, x\in E$  and $s,t \in [0,\infty)$,
$$\vp(0,\o)x=x \  \text{and}  \  \vp(t+s,\o)=\vp(t,\t_{s}\o)\circ\vp(s,\o).$$
If $x\mapsto\vp(t,\o)x$ is continuous for all $t\in\R$, $\o\in\O$,  then $\vp$ is a continuous RDS. 
\smallskip

A family $\{D(\o)\}_{\o\in\Omega}$ of non-empty subsets of $E$ is called a random closed (resp. compact) set if it is $\P$-a.s. closed (resp. compact) and  $\mathcal{F}$-measurable, that is, for each $x\in E$, 
   $$\o\mapsto d(x,D(\o)):=\inf_{y\in D(\o)}d(x,y) \ \text{is $\mathcal{F}$-measurable}, $$
and is called $\varphi$-invariant, if for all $t\ge0$ and a.s. in  $\o\in\Omega$,
 $$\vp(t,\o)D(\o) = D(\t_t\o).$$
A random, compact set $A$ is called a pullback random attractor of the RDS $\varphi$, if 
$A$ is $\varphi$-invariant, and for every compact set $B$ in $E$ and a.s.,
\begin{equation}\label{eq:pullback}
  \lim_{t\to\infty}\sup_{x\in B}d(\varphi(t,\theta_{-t}\omega)x,A(\omega))=0. 
\end{equation}
If \eqref{eq:pullback} is replaced by
  $$  \lim_{t\to\infty}\sup_{x\in B}d(\varphi(t,\omega)x,A(\theta_{t}\omega))=0,$$
then $A$ is said to be a forward random attractor.

\section{The proof of Lemma~\ref{lem:integral_estimate}}\label{lemma}

We present here the proof of Lemma~\ref{lem:integral_estimate}.

\begin{proof}
We compute the Fourier transform $\hat{\phi} $ of $\phi$. Using the elementary fact that $\int e^{-2\pi i z\cdot w} e^-\frac{|w|^2}{a} dw= \sqrt{a\pi}e^{-a\pi^{2}|z|^{2}}$, we find
\begin{align*}
\hat{\phi}(z) & =\int e^{-2\pi iz\cdot w}\phi(w)dw =\int e^{-2\pi iz\cdot w}e^{-\frac{|w|^{2}}{a}}\int e^{ib(\xi)\cdot w}f(\xi)d\xi dw\\
 & =\int\int e^{-2\pi i(z-\frac{1}{2\pi}b(\xi))\cdot w}e^{-\frac{|w|^{2}}{a}}dwf(\xi)d\xi
 =\int\sqrt{a\pi}e^{-|\sqrt{a}\pi(\frac{1}{2\pi}b(\xi)-z)|^{2}}f(\xi)d\xi,
\end{align*}
and, hence, 
\begin{align*}
|\hat{\phi}(z)|^{2} & =\left|\int\sqrt{a\pi}e^{-|\sqrt{a}\pi(\frac{1}{2\pi}b(\xi)-z)|^{2}}f(\xi)d\xi\right|^{2}\\
 & \le a\pi\int e^{-|\sqrt{a}\pi(\frac{1}{2\pi}b(\xi)-z)|^{2}}d\xi\int e^{-|\sqrt{a}\pi(\frac{1}{2\pi}b(\xi)-z)|^{2}}f^{2}(\xi)d\xi.
\end{align*}

Next we use the assumption on $b$ which enters in the following straightforward estimate:
\begin{align*}
\int e^{-|\sqrt{a}\pi(\frac{1}{2\pi}b(\xi)-z)|^{2}}d\xi & =\int_0^\infty  2\tau e^{-\tau{}^{2}}|\{\xi: \sqrt{a}\pi|\frac{1}{2\pi}b(\xi)-z)|<\tau\}|d\tau\\
= & \int_0^\infty  2\tau e^{-\tau^{2}}|\{\xi :|b(\xi)-2\pi z|<\frac{2\tau}{\sqrt{a}}\}|d\tau\\
\le & \int_0^\infty  2\tau e^{-\tau^{2}}\iota(\frac{2\tau}{\sqrt{a}})d\tau.
\end{align*}
Hence
\begin{align*}
|\hat{\phi}(z)|^{2} & \le2a\pi\int_{0}^{\infty}\tau e^{-\tau^{2}}\iota(\frac{2\tau}{\sqrt{a}})d\tau\int e^{-|\sqrt{a}\pi(\frac{1}{2\pi}b(\xi)-z)|^{2}}f^{2}(\xi)d\xi.
\end{align*}
\smallskip

Integrating the above inequality in $z$ and using that  
$$\int e^{-|\sqrt{a}\pi(\frac{1}{2\pi}b(\xi)-z)|^{2}}dz=\int e^{-a\pi^{2}|z|^{2}}dz=
\frac{1}{\sqrt{a\pi}}$$ yields
\begin{align*}
\int|\hat{\phi}(z)|^{2}dz & \le2\sqrt{a\pi}\int_{0}^\infty \tau e^{-\tau^{2}}\iota(\frac{2\tau}{\sqrt{a}})d\tau\|f\|_{2}^{2}.
\end{align*}
\end{proof}

\section{The derivation of equation (\ref{eq:splitup})}\label{app:splitup}
We give here the details about the derivation of \eqref{eq:splitup}. Since $u\in C([0,T];L^1(\TT^N))$, the definition of pathwise entropy solutions yields that, for all $t\ge 0$, all $\vr$ given by \eqref{takis10000} with $\vr_{0}\in C^{\infty}(\TT^N)$ and all  $\vp\in C^{1}(\TT^N \times [0,t])$, 
\begin{align*}
 & \int_{0}^{t}\int_{\R\times\TT^N}\partial_{s}\vp(y,s)(\vr\ast\chi)(y,\xi,s)dyd\xi ds+\int_{\R\times\TT^N}\vp(y,0)(\vr\ast\chi)(y,\xi,0)dyd\xi\\
 & = \int_{\R\times\TT^N}\vp(y,t)(\vr\ast\chi)(y,\xi,t)dyd\xi \\
 &\quad+ \int_{\TT^N} \int_{0}^{t}\int_{{\TT^N}\times \R}\partial_{\xi}\big(\vp(y,s)\vr(x,y,\xi,s)\big)dm(x,\xi,s)dy.
\end{align*}
Replacing $\vp$ by $\vp(y- \boldsymbol{a}(\xi)\cdot \boldsymbol{\b}_t,s)$ implies 
\begin{align*}
 & \int_{0}^{t}\int_{\R\times\TT^N}\partial_{s}\vp(y,s)\vr_0(x-y-\boldsymbol{a}(\xi)\cdot (\boldsymbol{\b}_s- \boldsymbol{\b}_t))\chi(x,\xi,s)dxdyd\xi ds\\
 &+ \int_{\R\times\TT^N}\int_{\TT^N}\vp(y,0)\vr_0(x-y-\boldsymbol{a}(\xi)\cdot (\boldsymbol{\b}_0- \boldsymbol{\b}_t))\chi(x,\xi,0)dxdyd\xi \\
 & = \int_{\R\times\TT^N}\int_{\TT^N}\vp(y,t)\vr_0(x-y)\chi(x,\xi,t)dxdyd\xi\\
 &+ \int_{\TT^N} \int_{0}^{t}\int_{{\TT^N}\times \R}\partial_{\xi}\big(\vp(y,s)\vr_0(x-y-\boldsymbol{a}(\xi)\cdot (\boldsymbol{\b}_s- \boldsymbol{\b}_t))\big)dm(x,\xi,s)dy.
\end{align*}
Given  $\vp \in C^1(\TT^N)$ we use next $S^*_{\g B}(t-\cdot)\vp$ as a test function in the formula above. Since
  $$\partial_s (S^*_{\g B}(t-s))\vp =\g B (S^*_{\g B}(t-s))\vp,$$
we obtain
\begin{align*}
 & \int_{0}^{t}\int_{\R\times\TT^N} \g B (S^*_{\g B}(t-s)\vp)(y)\vr_0(x-y-\boldsymbol{a}(\xi)\cdot (\boldsymbol{\b}_s- \boldsymbol{\b}_t))\chi(x,\xi,s)dxdyd\xi ds\\
 &+\int_{\R\times\TT^N}\int_{\TT^N}(S^*_{\g B}(t)\vp)(y)\vr_0(x-y-\boldsymbol{a}(\xi)\cdot (\boldsymbol{\b}_0- \boldsymbol{\b}_t))\chi(x,\xi,0)dxdyd\xi \\
 & = \int_{\R\times\TT^N}\int_{\TT^N}\vp(y)\vr_0(x-y)\chi(x,\xi,t)dxdyd\xi\\
 &+ \int_{\TT^N} \int_{0}^{t}\int_{{\TT^N}\times \R}\partial_{\xi}\left((S^*_{\g B}(t-r)\vp)(y)\vr_0(x-y-\boldsymbol{a}(\xi)\cdot (\boldsymbol{\b}_s- \boldsymbol{\b}_t))\right)dm(x,\xi,s)dy,
\end{align*}
and, in view of \eqref{eq:semigroup},
\begin{align*}
 & \int_{0}^{t}\int_{\R\times\TT^N} \g B (S^*_{A_\g(\xi)}(s,t)\vp)(y)\vr_0(x-y)\chi(x,\xi,s)dxdyd\xi ds\\
 &+\int_{\R\times\TT^N}\int_{\TT^N}(S^*_{A_\g(\xi)}(0,t)\vp)(y)\vr_0(x-y)\chi(x,\xi,0)dxdyd\xi \\
 &= \int_{\R\times\TT^N}\int_{\TT^N}\vp(y)\vr_0(x-y)\chi(x,\xi,t)dxdyd\xi\\ 
 &+ \int_{\TT^N} \int_{0}^{t}\int_{{\TT^N}\times \R}\partial_{\xi}\left((S^*_{A_\g(\xi)}(s,t)\vp)(y)\right)\vr_0(x-y)dm(x,\xi,s)dy.
\end{align*}
Let $\vr_0^\ve$ be a approximation of the identity. Letting $\ve\to 0$  in the above equation with $\vr_0$ replaced by  $\vr_0^\ve$  yields
\begin{align*}
 \int_{\R\times\TT^N}\vp(x)\chi(x,\xi,t)dxd\xi
 &=\int_{\R\times\TT^N}(S^*_{A_\g(\xi)}(0,t)\vp)(x)\chi(x,\xi,0)dxd\xi\\
 &+\int_{0}^{t}\int_{\R\times\TT^N} \g B (S^*_{A_\g(\xi)}(s,t)\vp)(x)\chi(x,\xi,s)dxd\xi ds\\
 &- \int_{0}^{t}\int_{{\TT^N}\times \R}\partial_{\xi}\left((S^*_{A_\g(\xi)}(s,t)\vp)(x)\right)dm(x,\xi,s),
\end{align*}
which finishes the proof.

\bibliographystyle{acm}
\bibliography{refs}

\end{document}